\documentclass[11pt]{amsart}
\usepackage{amsmath, amsfonts,amsthm, amssymb,amscd, verbatim, graphicx,color,multirow,booktabs,tikz}
\usepackage{footnote}
\makesavenoteenv{tabular}
\makesavenoteenv{table}
\usepackage{array,etoolbox}
\preto\tabular{\setcounter{magicrownumbers}{0}}
\newcounter{magicrownumbers}

\usepackage{chngcntr}
\counterwithin{table}{section}
\def\classification#1{\def\@class{#1}}
\classification{\null}
\newcounter{rowcount}
\usepackage[margin=1.5in]{geometry}
\usepackage{xcolor,tkz-berge}
\newcommand{\Out}{\mathop{\mathrm{Out}}}

\usepackage{hyperref}

\newcommand{\Alt}{\mathop{\mathrm{Alt}}}
\newcommand{\fix}{{\mathrm{fix}}}
\newcommand{\Sym}{\mathop{\mathrm{Sym}}}
\newcommand{\Aut}{\mathop{\mathrm{Aut}}}
\renewcommand{\wr}{\mathop{\mathrm{wr}}}
\def\nor#1#2{{\bf N}_{{#1}}{{(#2)}}}
\def\centr#1#2{{\bf C}_{{#1}}{{(#2)}}}
\DeclareFontFamily{OT1}{rsfs}{}
\DeclareFontShape{OT1}{rsfs}{n}{it}{<-> rsfs10}{}
\DeclareMathAlphabet{\mathscr}{OT1}{rsfs}{n}{it}
 
\newcommand{\Fq}{\mathbb{F}_q}
\newcommand{\K}{\mathbb{K}}

\newcommand{\Zplus}{\mathbb{Z}^{+}}

\newcommand{\C}{\mathcal{C}}
\newcommand{\B}{\mathcal{B}}

\newcommand{\stab}{\mathrm{Stab}}
\newcommand{\SL}{\mathrm{SL}}
\newcommand{\SU}{\mathrm{SU}}
\newcommand{\SOr}{\mathrm{SO}}
\newcommand{\AGL}{\mathrm{AGL}}

\newcommand{\Sp}{\mathrm{Sp}}
\newcommand{\POmega}{\mathrm{P\Omega}}
\newcommand{\PSL}{\mathrm{PSL}}
\newcommand{\PSU}{\mathrm{PSU}}
\newcommand{\PGL}{\mathrm{PGL}}
\newcommand{\Ree}{{^2\mathrm{G}_2}}

\newcommand{\GL}{\mathrm{GL}}
\newcommand{\GU}{\mathrm{GU}}
\newcommand{\symme}{\mathrm{Sym}}
\newcommand{\alter}{\mathrm{Alt}}

\newtheorem{prop}{Proposition}[section]

\newtheorem*{thma}{Theorem A}
\newtheorem*{thmb}{Theorem B}
\newtheorem{conj}[prop]{Conjecture}

\newtheorem{cor}[prop]{Corollary}
\newtheorem{lem}[prop]{Lemma}

\theoremstyle{definition}

\numberwithin{equation}{section}
\def\Norm#1#2{{\bf N}_{{#1}}({{#2}})}
\begin{document}

\title{Binary permutation groups: alternating and classical groups}

\author{Nick Gill}
\address{ Department of Mathematics, University of South Wales, Treforest, CF37 1DL, U.K.}
\email{nick.gill@southwales.ac.uk}

\author{Pablo Spiga}
\address{Dipartimento di Matematica e Applicazioni, University of Milano-Bicocca, Via Cozzi 55, 20125 Milano, Italy} 
\email{pablo.spiga@unimib.it}

\begin{abstract}
We introduce a new approach to the study of finite binary permutation groups and, as an application of our method,
  we prove Cherlin's binary groups conjecture for groups with socle a finite alternating group, and for the $\mathcal{C}_1$-primitive actions of the finite classical groups.

Our new approach involves the notion, defined with respect to a group action, of a `\emph{beautiful subset}'. We demonstrate how the presence of such subsets can be used to show that a given action is not binary. In particular, the study of such sets will lead to a resolution of many of the remaining open cases of Cherlin's binary groups conjecture.
\end{abstract}

\maketitle

\section{Introduction}\label{s: intro}

In this paper, the following conjecture is our main concern \cite{cherlin1}.

\begin{conj}\label{conj: cherlin original}
A finite primitive binary permutation group must be one of the following:
\begin{enumerate}
 \item a symmetric group $\symme(n)$ acting naturally on $n$ elements;
\item  a cyclic group of prime order acting regularly on itself;
\item an affine orthogonal group $V\cdot {\rm O}(V)$ with $V$ a vector space over a finite field equipped
with an anisotropic quadratic form, acting on itself by translation, with complement
the full orthogonal group ${\rm O}(V)$.
\end{enumerate}
\end{conj}

For readers unfamiliar with the notion of a ``binary permutation group'', we refer to \S\ref{s: background}, where all necessary terminology can be found. Note that in that section, we use the definition from~\cite{cherlin_martin} which is couched in terms of group actions -- this definition is useful, and elementary, but lacking in any obvious motivation. For an equivalent definition in terms of {\it relational structures}, we refer the reader to the opening paragraphs of \cite{cherlin2}; when seen from this point of view, the concept of a binary permutation group becomes a very natural one indeed.

A remark concerning part~(3) of Conjecture~\ref{conj: cherlin original}: Let $V$ be a vector space of dimension $d$ over a finite field $\mathbb{F}_q$ with $q$ elements. Since the quadratic form on $V$ is anisotropic, $\dim_{\mathbb{F}_q}(V)$ is either $1$ or $2$. In the case where $\dim_{\mathbb{F}_q}(V)=1$, ${\rm O}(V)$ is cyclic of order $2$ and hence $V\cdot{\rm O}(V)$ is dihedral of order $2q$ and $q$ is a prime number. In the case where $\dim_{\mathbb{F}_q}(V)=2$, ${\rm O}(V)$ is dihedral of order $2(q+1)$.

Cherlin himself gave a proof of this conjecture when $G$ is ``of affine type'', i.e. $G$ has an elementary abelian socle.  Wiscons then studied some other cases and showed that Conjecture~\ref{conj: cherlin original} reduces to the following statement concerning almost simple groups \cite{wiscons}.

\begin{conj}\label{conj: cherlin}
 If $G$ is a finite binary almost simple primitive group on $\Omega$, then $G=\symme(\Omega)$. 
\end{conj}

Our first main result is a proof of Conjecture~\ref{conj: cherlin} for one of the main families of finite simple groups:

\begin{thma}
Let $G$ be a finite almost simple primitive group on $\Omega$ with socle a non-abelian alternating group. If $G$ is binary, then $G=\Sym(\Omega)$.
\end{thma}

In the light of Theorem A, to prove Conjecture~\ref{conj: cherlin} we must deal with the sporadic simple groups and the finite simple groups of Lie type. Our second main result concerns the second of these two families:

\begin{thmb}
 Let $G$ be a finite almost simple classical group, and let $M$ be a maximal subgroup of $G$ that lies in the Aschbacher family $\mathcal{C}_1$. Let $\Omega$ be the set of (right) cosets of $M$. If the action of $G$ on $\Omega$ is binary, then $G\cong \PGL_2(4)\cong \symme(5)$, $|\Omega|=5$ and $M\cong \symme(4)$.
\end{thmb}

Note that we use the description of the Aschbacher family $\mathcal{C}_1$ given by Kleidman and Liebeck~\cite{kl}, rather than that used in the original paper \cite{aschbacher2}. In particular, this theorem excludes those almost simple groups $G$ which have socle $\POmega_8^+(q)$ and include a triality automorphism, and those which have socle $\Sp_4(2^f)$ and include a graph automorphism.

Both Theorems A and B are relatively easy consequences of results stated later in the paper (Proposition~\ref{p: alt} and Table~\ref{t: c1 sln} in \S\ref{s: classical}) concerning the existence of something we call a {\em beautiful subset}. To define it note that, for $\Lambda$ a subset of $\Omega$, we write $G_\Lambda$ for the \emph{set-wise} stabilizer of $\Lambda$, $G_{(\Lambda)}$ for the \emph{point-wise} stabilizer of $\Lambda$ and $G^\Lambda$ for the permutation group induced by $G_\Lambda$ on $\Lambda$. 

We will say that a subset $\Lambda\subseteq \Omega$ is a \emph{$G$-beautiful subset} if $G^\Lambda$ is a 2-transitive subgroup of $\symme(\Lambda)$ that is isomorphic to neither $\alter(\Lambda)$ nor $\symme(\Lambda)$. If the group $G$ is clear from the context, we will speak of a \emph{beautiful subset} rather than a $G$-beautiful subset of $\Omega$. Observe that any beautiful subset of $\Omega$ must have cardinality at least $5$.

Now the point is that, if $\Omega$ contains a $G$-beautiful subset, then the action of $G$ on $\Omega$ is not binary; indeed, one can say more: if $G$ is almost simple with socle $S$ and if $\Omega$ contains an $S$-beautiful subset, then the action of $G$ on $\Omega$ is not binary (see Lemma~\ref{l: 2trans gen}). 

The results from which we deduce Theorems A and B all assert that, for the actions under consideration, beautiful subsets are present except in very special circumstances; what is more these special circumstances are completely classified. For instance, the results in \S\ref{s: classical} have the following immediate corollary:

\begin{cor}
 Let $G$ be a finite  almost simple classical group with socle $S$, and let $M$ be a maximal subgroup of $G$ that lies in the Aschbacher family $\mathcal{C}_1$. Let $\Omega$ be the set of (right) cosets of $M$. Then, with finitely many exceptions, $\Omega$ contains an $S$-beautiful subset.
\end{cor}

The ``finitely many exceptions'' are all explicitly listed in Table~\ref{t: c1 sln}.

\subsection{Implications, and generalizations}

As well as representing a material advance towards a proof of Conjecture~\ref{conj: cherlin}, Theorems A and B also demonstrate the efficacy of studying beautiful subsets for primitive actions. 

Note, first of all, that this method will undoubtedly allow further cases of Conjecture~\ref{conj: cherlin} to be proved. In \S\ref{s: classical method}, we give further details concerning how this method can be used for the ``algebraic actions'' of the finite classical and exceptional almost simple groups. A paper concerning these actions is in preparation \cite{gls_binary}.

A slight variant of this method can also be applied to the study of the sporadic simple groups. Indeed, it is possible to prove a relatively strong upper bound on the number of orbits of a binary primitive group in its action on $3$-tuples in terms of its rank. This bound, combined with information from~\cite{ATLAS}, allows us to show that most primitive actions of sporadic simple groups are not binary. For the remaining actions, we may then use detailed information on the subgroup structure and apply the method of beautiful subsets.

The notion of a beautiful subset is a particular example of something we call a \textit{forbidden configuration}. This more general concept can be used to prove cases of Cherlin's conjecture where beautiful subsets do not exist \cite{ghs_binary}; it also has implications for the study of actions of arities other than $2$. A discussion of this idea can be found below in \S\ref{s: forbidden}; in particular, we give details of a relatively straightforward way to strengthen Theorem~B.

It is too soon to say whether or not a complete proof of Conjecture~\ref{conj: cherlin} will be obtained using the approach of beautiful subsets and forbidden configurations, although this is our ultimate goal.

\subsection{Structure and acknowledgments}

In \S\ref{s: background}, we set up the terminology that will be used throughout the paper, and we prove a number of basic lemmas. We also give some perspective on our approach to the study of binary groups, with a view to applying the basic ideas in later papers. 

In \S\ref{s: alternating}, we give a full classification of those primitive actions of an alternating group that admit a beautiful subset (Proposition~\ref{p: alt}), and we use this result to prove Theorem~A. In \S\ref{s: classical}, we do similarly for the actions of classical groups on the cosets of $\mathcal{C}_1$-maximal subgroups.

It is a pleasure to thank Martin Liebeck for introducing us to Cherlin's conjecture and for a number of enlightening conversations. Thanks are also due to Francis Hunt and Josh Wiscons for useful discussions.

\section{Background lemmas}\label{s: background}

In this section, we establish some basic definitions and lemmas that will be used throughout the paper. Note that in this section, $G$ is a subgroup of $\symme(\Omega)$, where $\Omega$ is a finite set. 

The group $G$ is called \textit{$r$-subtuple complete} with respect to a
pair of $n$-tuples $I, J \in \Omega^n$, if it contains elements that
map every subtuple of size $r$ in $I$ to the corresponding subtuple in
$J$ i.e. $\forall k_1, k_2, \dots, k_r$ integers with $1\leq k_1, k_2,
\dots, k_r\leq n \; \exists h \in G$ with $I_{k_i}^h=J_{k_i}$ for $i =
1, \dots, r$. Here $I_k$ denotes the $k^{\text{th}}$ element of tuple
$I$ and $I^g$ denotes the image of $I$ under the action of $g$.
Note that $n$-subtuple completeness simply requires the existence of
an element of $G$ mapping $I$ to $J$.

The group $G$ is said to be of {\it arity $r$} if, for all
$n$-tuples $I, J \in \Omega^n$, with $n\in\Zplus$ and $n\geq r$, $r$-subtuple
completeness implies $n$-subtuple completeness. When $G$ has arity 2, we say that $G$ is {\it binary}.\footnote{We briefly mention an alternative approach to the notion of binari-ness, that derives from ideas found in~\cite{babai}: a
subset $\Sigma \subseteq \Omega$ is said to {\it split the pair} $\{\alpha,\beta\}$
(with $\alpha,\beta\in\Omega$) if $\alpha$ and $\beta$ belong to 
distinct $G_\Sigma$-orbits. It is easy to see that  $G$ is binary if and only if, for every $n\in\mathbb{Z}^+$, for every $\Sigma=\{\alpha_1,\dots,\alpha_{n}\}\subseteq \Omega$ and for every $\gamma,\delta\in \Omega$, the set $\Sigma$ splits  $\{\gamma,\delta\}$ only when, for each $i$, $\{\alpha_i\}$ does too.}
A pair of $n$-tuples of $\Omega$ is called a {\it non-binary witness of size $n$} if it is $2$-subtuple complete, but not $n$-subtuple complete.

The first lemma is well-known, and is an easy exercise. It is included as motivation for Lemma~\ref{l: 2trans gen}.

\begin{lem}\label{l: 2trans}
If $G$ is binary and $2$-transitive, then $G=\symme(\Omega)$.
 \end{lem}

Recall that we write $G^{\Delta}$ for the permutation group induced by the set-wise stabilizer $G_{\Delta}$ on $\Delta$: clearly, $G^{\Delta}\cong G_{\Delta}/G_{(\Delta)}$.
 
\begin{lem}\label{l: 2trans gen}
 Suppose that there exists a subset $\Lambda \subseteq \Omega$ such that $|\Lambda|\geq 2$ and $G^\Lambda$ is a $2$-transitive proper subgroup of $\symme(\Lambda)$. Then $G$ is not binary.
\end{lem}

Note that the condition $|\Lambda|\geq 2$ is included so that the phrase ``2-transitive'' makes sense. Of course, $\symme(\Lambda)$ can only contain a $2$-transitive proper subgroup  when $|\Lambda|\geq 4$.

\begin{proof}
Let $|\Lambda|=k$ and observe that, by supposition, any pair of $k$-tuples of distinct elements from $\Lambda$ is 2-subtuple complete. If $G$ were binary, this would imply that $G_\Lambda$ induces every permutation of $\symme(\Lambda$), i.e. $G^\Lambda=\symme(\Lambda)$, which is a contradiction.
\end{proof}

The next corollary is a slight strengthening of Lemma~\ref{l: 2trans gen} expressed in the language of beautiful subsets, the definition of which was given in \S\ref{s: intro}.

\begin{cor}\label{c: 2trans gen}
 Suppose that $G$ is almost simple with socle $S$. If $\Omega$ contains an $S$-beautiful subset, then $G$ is not binary.
\end{cor}
\begin{proof}
 Let $\Lambda$ be an $S$-beautiful subset of $\Omega$ and observe that $\Lambda$ has order at least $5$. Since $S$ is normal in $G$, the group $(S_\Lambda G_{(\Lambda)})/G_{(\Lambda)}$ is a normal subgroup of $G_\Lambda/G_{(\Lambda)}$, that is, $S^\Lambda\unlhd G^\Lambda$. Since $\Lambda$ is $S$-beautiful, we get $\Alt(\Lambda)\nleq S^\Lambda$, and hence $G^\Lambda$ is a $2$-transitive proper subgroup of $\symme(\Lambda)$. Then Lemma~\ref{l: 2trans gen} implies that $G$ is not binary.
\end{proof}

The next lemma will be useful when we search for beautiful subsets. (Recall that $G$ is a subgroup of $\symme(\Omega)$.)

\begin{lem}\label{l: 2t}
Let $K$ be a group endowed with a  $2$-transitive action, and let $K_0$ be a point-stabilizer in $K$.
Let $H$ be a subgroup of $G$ and suppose that $\varphi:H\to K$ is an isomorphism. Let $M$ be the stabilizer in $G$ of a point $\omega\in\Omega$ and suppose that $\varphi(H\cap M)=K_0$. Then $H$ acts $2$-transitively on the orbit $\omega^H$.
\end{lem}
\begin{proof}
 The group isomorphisms $\varphi:H\to K$ and $\varphi_{|H\cap M}:H\cap M\to K_0$ induce a bijection from the set $K/K_0$ of right cosets of $K_0$ in $K$ to the set $H/(H\cap M)$ of right cosets of $H\cap M$ in $H$. Thus, we obtain a permutation isomorphism, and the result follows.
\end{proof}

Note that we do not require that $K$ is \emph{faithful} in Lemma~\ref{l: 2t}. We will tend to use the lemma with $2$-transitive groups $K$ where $K/N\cong E_{p^a}\rtimes C_{p^a-1}$. (Here we write $N$ for the kernel of the action of $K$, $E_{p^a}$ for the elementary abelian group of order $p^a$, and $C_{p^a-1}$ for the cyclic group of order $p^a-1$.) The group $K$ is sharply $2$-transitive on a set of size $p^a$ with stabilizer cyclic of order $C_{p^a-1}$. 

The next lemma will be useful when we come to deal with those actions that do not contain a beautiful subset.

\begin{lem}\label{l: auxiliary}
Let  $\omega_0,\omega_1,\omega_2\in \Omega$ with $G_{\omega_0}\cap G_{\omega_1}=1$. Suppose that there exists $g\in G_{\omega_0}\cap G_{\omega_2}G_{\omega_1}$ with $g\notin G_{\omega_2}$. Then the two triples $(\omega_0,\omega_1,\omega_2)$ and $(\omega_0,\omega_1,\omega_2^g)$ are $2$-subtuple complete but are not $3$-subtuple complete. In particular, $G$ is not binary.
\end{lem}
\begin{proof}
  Let $x\in G_{\omega_2}$ and $y\in G_{\omega_1}$ with $g=xy$. A computation shows that $(\omega_0,\omega_1)^{id_G}=(\omega_0,\omega_1)$, $(\omega_0,\omega_2)^{g}=(\omega_0,\omega_2^g)$ and $(\omega_1,\omega_2)^{y}=(\omega_1,\omega_2^g)$, and hence the two triples $(\omega_0,\omega_1,\omega_2)$ and $(\omega_0,\omega_1,\omega_2^g)$ are $2$-subtuple complete. If $(\omega_0,\omega_1,\omega_2)$ and $(\omega_0,\omega_1,\omega_2^g)$ are $3$-subtuple complete, then there exists $z\in G$ with $(\omega_0,\omega_1,\omega_2)^z=(\omega_0,\omega_1,\omega_2^g)$; in particular, $z\in G_{\omega_0}\cap G_{\omega_1}=1$ and hence $g\in G_{\omega_2}$, which is a contradiction.
\end{proof}

The final lemma is nothing more than an observation. It requires no proof but is included as it will be used repeatedly in what follows.

\begin{lem}\label{l: stabs}
 Let $H$ be a subgroup of $G$, let $M$ be the stabilizer in $G$ of the point $\omega\in\Omega$ and let $\Lambda=\omega^H$. Then $\centr MH\leq G_{(\Lambda)}$ and $G_\Lambda \leq \nor G {G_{(\Lambda)}}$.
\end{lem}

\subsection{Forbidden configurations}\label{s: forbidden}

As we mentioned above, Lemma~\ref{l: 2trans gen} can be thought of as a particular example of a \textit{forbidden configuration}. The idea is as follows: suppose we want to prove that the action of $G$ on $\Omega$ is not binary; to do this it would be convenient to be able to observe that there is some subset $\Lambda\subseteq \Omega$ for which $G^\Lambda$ is not binary, and conclude (somehow) that the action of $G$ on $\Omega$ is not binary. Unfortunately, the fact that $G^\Lambda$ is not binary does not necessarily imply that $G$ is not binary -- but a moment's thought shows that this implication {\bf will} hold if we can find a non-binary witness in $\Lambda$ of size $|\Lambda|$. Let us summarise this observation:

\begin{lem}\label{l: forbidden}
 Suppose that there exists a subset $\Lambda \subseteq \Omega$ for which $G^\Lambda$ is  not binary and for which there exists a non-binary witness of size $|\Lambda|$. Then $G$ is not binary.
\end{lem}

Now, Lemma~\ref{l: 2trans gen} is a particular case of Lemma~\ref{l: forbidden}. More concretely let us define a {\it forbidden configuration} to be a subset $\Lambda\subseteq \Omega$ for which $G^\Lambda$ is not binary and for which there is a non-binary witness in $\Lambda$ of size $|\Lambda|$. 

In this paper, the only example of a forbidden configuration that will need is that of a beautiful subset. However other examples do exist: Suppose that $\Lambda$ is a subset of $\Omega$ of size $6$. Let us write $\Lambda=\{1,2,3,4,5,6\}$. Now suppose that $G^\Lambda$ contains the elements $f=(1\,2)(3\,4)$ and $g=(1\,2)(5\,6)$ but not the element $h=(1\,2)$. Then one can see immediately that the following pair is a non-binary witness of size $6$:
\[
((1,2,3,4,5,6),(2,1,3,4,5,6)).
\]
This forbidden configuration has already proved useful in the study of the primitive actions of almost simple groups with socle isomorphic to ${\rm PSL}_2(q)$, for some prime power $q$. These actions do not, in general, admit beautiful subsets so that line of attack on Cherlin's conjecture is closed to us; however the forbidden configuration described above occurs for many of these actions, and immediately yields a proof of Cherlin's conjecture in this setting \cite{ghs_binary}.

One is naturally led to ask the following question: \emph{Can one classify the pairs $(H,\Lambda)$ where $H$ is a group acting on a set $\Lambda$ such that there is a non-binary witness of size $|\Lambda|$?} 

Note, finally, that we have defined forbidden configurations with respect to the study of binary actions. It should be clear that the notion can be easily extended to the study of actions of any fixed arity, and could be used in an attempt to prove a version of Conjecture~\ref{conj: cherlin original} for primitive actions of arity $3$ or $4$.

A concrete example can be found in our approach to Theorem~B: our method there is to find a beautiful subset $\Lambda$ of size $t$ for almost every action of a simple group $S$ under the ambit of the theorem. In a number of cases, we have $t=q+1$ (where $q$ is a prime power) and we demonstrate that the action of $S^\Lambda$ on $\Lambda$ is permutation isomorphic to the natural action of $\PGL_2(q)$ on the $q+1$ points of the projective line over $\Fq$. In particular, the action of $S^\Lambda$ is 3-transitive and hence, for $q\geq 5$, the subset $\Lambda$ constitutes a forbidden configuration for actions of arity $3$ as well as for binary actions. One could use this approach in many other cases to prove that the arity of the actions considered in Theorem~B is not just greater than $2$, as the theorem asserts, but greater than $3$ (with more exceptions).

\section{Alternating socle}\label{s: alternating}

\begin{prop}\label{p: alt}
Let $G$ be a finite almost simple primitive group on $\Omega$ with socle the alternating group $\Alt(n)$ with $n\geq 5$. Then one of the following holds:
\begin{enumerate}
 \item $\Omega$ contains a beautiful subset (with respect to this action);
 \item the action is one of those listed in Table~$\ref{t: alt}$.
\end{enumerate}
 \end{prop}
\begin{center}
\begin{table}[!ht]
\begin{tabular}{ccl}
\toprule[1.5pt]
Line & Group & Details of action \\
\midrule[1.5pt]
1 & $\Alt(n)$ or $\Sym(n)$, $n\geq 5$ & Natural action: $\Omega=\{1,\dots, n\}$; \\
2 & $\Alt(n)$ or $\Sym(n)$, $n\geq 5$ & Action on $2$-subsets: $\Omega=\{\{i,j\} \mid 1\leq i<j\leq n\}$; \\
3 & $\Alt(n)$, $n\ge 13$ prime&Action on right cosets of $\AGL_1(n)\cap \Alt(n)$;\\		 
4 & $\Alt(n)$ or $\Sym(n)$&Action on right cosets of a max. subgroup $M$,\\
&&$M$ primitive on $\{1,\ldots,n\}$ and with socle\\
&&$\PSL_2(q)$ where $q=2^f$ or $q\equiv 3,5\pmod 8$.\\
5&$\Alt(6)$ or $\Sym(6)$&actions of degree $6$ or $15$\\
\bottomrule[1.5pt]
\end{tabular}
\caption{Primitive actions of $\Alt(n)$ and $\Sym(n)$ -- Cases where a beautiful subset may not exist.}\label{t: alt}
\end{table}
\end{center}

%Note that the statement of Proposition~\ref{p: alt} covers all almost simple groups $G$ with an alternating socle, except when the socle is $\alter(6)$ and $G$ is $M_{10}$, $\mathrm{P}\Sigma\mathrm{L}_2(9)$ and $\PGammaL_2(9)$. These latter cases are covered by unpublished calculations of Wiscons \cite{wiscons2} using GAP \cite{GAP}.

The case where $G$ has socle $\Alt(6)$ is exceptional. Note first that the proposition implies that, in this case, all primitive actions of $G$ have a beautiful subset except when $G$ is either $\Alt(6)$ or $\Sym(6)$. In these cases, $G$ has two conjugacy classes of maximal subgroups isomorphic to $\Alt(5)$ or to $\Sym(5)$. The two primitive actions corresponding to these two conjugacy classes are permutation isomorphic to each other, and are permutation isomorphic to the natural action of $G$; similarly $G$ has two conjugacy classes of maximal subgroups isomorphic to $\Sym(4)$ or to $\Sym(4)\times C_2$; again the corresponding primitive actions are permutation isomorphic to each other, and are permutation isomorphic to the action of $G$ on $2$-subsets. Thus, reading Table~\ref{t: alt} up to permutation isomorphism (rather than permutation equivalence), Line~5 is superfluous. 

The groups appearing in the first two lines of Table~\ref{t: alt} are genuine exceptions for Proposition~\ref{p: alt}, see Lemma~\ref{l: M intrans 2}. However, some of the groups in lines~$3$ or~$4$ may have beautiful subsets. For instance, $\Alt(41)$ acting on the cosets of $\Alt(41)\cap \AGL_1(41)$ does have a beautiful subset of size $5$; whereas, $\Alt(13)$ acting on the cosets of $\Alt(13) \cap \AGL_1(13)$ has no beautiful subsets. A similar comment applies for the groups in the fourth line of Table~\ref{t: alt}. The existence of beautiful subsets for these actions depends on detailed arithmetical conditions and studying these here would take too much of a detour.

Before we prove Proposition~\ref{p: alt}, let us see why it implies Theorem~A. 

\begin{proof}[Proof of Theorem A]
We must deal with the actions given in Table~\ref{t: alt}.

\smallskip

\noindent\textsc{Line~1 in Table~\ref{t: alt}.}

\smallskip

\noindent If $G=\Sym(n)$ in its natural action, then the action is binary and the conjecture holds. If $G=\Alt(n)$ in its natural action, then the relational complexity is well-known to be equal to $n-1$ (see~\cite[page~$340$]{cherlin2}), and the action is not binary.

\smallskip

\noindent\textsc{Line~2 in Table~\ref{t: alt}.}

\smallskip

\noindent If $G=\Sym(n)$ and $G$ acts on the set of $2$-subsets of $\{1,\ldots,n\}$, then~\cite[Theorem~1, page~267]{cherlin_martin} asserts that the relational complexity is $3$, and so the action is not binary.

If $G=\Alt(n)$ and $G$ acts on the set of $2$-subsets of $\{1,\ldots,n\}$, then consider the following pair of $(n-1)$-tuples:
\begin{align*}
& (\{1,2\}, \{1,3\}, \{1,4\}, \{1,5\}, \{1,6\},\dots, \{1,n\}) \\
& (\{1,3\}, \{1,2\}, \{1,4\}, \{1,5\}, \{1,6\},\dots, \{1,n\}).
\end{align*}
One can map the first tuple to the second via the element $(2,3)\in \Sym(n)$. Moreover, using $n\geq 5$, one can quickly conclude that these two tuples are $2$-subtuple complete with respect to the action of $\Alt(n)$. On the other hand, $(2,3)$ is the only element in $\Sym(n)$ that maps the first tuple to the second, hence this pair of $(n-1)$-tuples is not $(n-1)$-subtuple complete. Thus the action of $\Alt(n)$ on $2$-sets is not binary.

\smallskip

\noindent\textsc{Line~3 in Table~\ref{t: alt}.}

\smallskip

\noindent Suppose now that $G=\Alt(n)$, $n$ is a prime number with $n\ge 13$ and $\Omega$ is the set of right cosets of $\AGL_1(n)\cap \Alt(n)$ in $\Alt(n)$. 
Consider $x:=(1\,2\,3\,4\,5 \cdots n)$ and $y:=x^{(1\,2\,3)}=(2\,3\,1\,4\,5\cdots n)$. An easy computation yields that $g:=xy=(2 \, 1\,3\,5\,7\,\cdots n\,4\,6\,8\,\cdots n\!\!-\!\! 1)$ is an $n$-cycle.

Now, $\nor G {\langle g\rangle}\cong \AGL_1(n)\cap \Alt(n)$ and the three subgroups $\nor G{\langle x\rangle}$, $\nor G {\langle y\rangle}$ and $\nor G{\langle g\rangle}$ are $G$-conjugate. Hence we may write $G_{\omega_0}=\nor G{\langle g\rangle}$, $G_{\omega_1}=\nor G{\langle y\rangle}$ and $G_{\omega_2}=\nor G{\langle x\rangle}$, for some $\omega_0, \omega_1,\omega_2\in \Omega$. In particular, $g\in G_{\omega_0}\cap G_{\omega_2}G_{\omega_1}$. A computation yields $G_{\omega_0}\cap G_{\omega_1}=G_{\omega_0}\cap G_{\omega_2}=1$. (For this to be true, we only need $n>5$.) Therefore, $g\notin G_{\omega_2}$ and, by Lemma~\ref{l: auxiliary}, $G$ is not binary.

\smallskip

\noindent\textsc{Line~4 in Table~\ref{t: alt}.}

\smallskip

\noindent Here $G$ is either $\Sym(n)$ or $\Alt(n)$ acting on $\Omega$, and for $\omega_0\in \Omega$, the point stabilizer $G_{\omega_0}$ is primitive  on $\Lambda = \{1,\ldots,n\}$ with socle $\PSL_2(2^f)$ for $f\ge 2$, or $\PSL_2(q)$ for $q\equiv 3,5\pmod 8$. As $\PSL_2(2^2)\cong\PSL_2(5)$, we assume that $f\ge 3$.

We first deal with small values of $q$. When $q=5$, $\PSL_2(5)$ has primitive permutation representations of degree $5$, $6$ and $10$. The action of degree $5$ gives $\PSL_2(5)\cong \Alt(5)$ and hence $\PSL_2(5)$ is not maximal in $\Alt(5)$. Similarly, the action of degree $10$ gives an embedding of $\PSL_2(5)$ into $\Alt(10)$ with $\PSL_2(5)$ not maximal in $\Alt(10)$. Finally, the action of degree $6$ gives rise to an embedding of $\PSL_2(5)$ into $\Alt(6)$, and the action of $\Alt(6)$ on the coset space $\Alt(6)/\PSL_2(5)$ is permutation isomorphic to the action of $\Alt(6)$ on $\{1,2,3,4,5,6\}$, which is not binary. The argument for $\PGL_2(5)$ is entirely similar. 

Next, $\PSL_2(8)$ and $\mathrm{P}\Gamma \mathrm{L}_2(8)$ have primitive permutation representations of degree $9$, $28$ and $36$. As $\PSL_2(8)<\mathrm{P}\Gamma \mathrm{L}_2(8)$ and $|\mathrm{P}\Gamma \mathrm{L}_2(8):\PSL_2(8)|=3$, we see that  in none of these permutation representations $\PSL_2(8)$ is maximal in the corresponding alternating group. Using  the inclusion $\mathrm{P}\Gamma \mathrm{L}_2(8)<\mathrm{PSp}_6(2)$, one can check that the permutation representations of degree $28$ and $36$ do not yield maximal subgroups for the corresponding alternating groups. Finally,  the permutation representation of $\mathrm{P}\Gamma \mathrm{L}_2(8)$ of degree $9$ gives rise to an embedding into $\Alt(9)$, now the action of $\Alt(9)$ on the cosets of $\mathrm{P}\Gamma \mathrm{L}_2(8)$ is primitive of degree $120$, and with the invaluable help of \texttt{magma}~\cite{magma} we check that this action is not binary. In the rest of the proof, we may assume that $q>8$.

Set
\[r:=
\begin{cases}
  \frac{q-1}{2}&\textrm{when }q\equiv 3\pmod 8,\\
  \frac{q+1}{2}&\textrm{when }q\equiv 5\pmod 8,\\
  q+1&\textrm{when }q=2^f.
  \end{cases}
\]
Observe that $r$ is odd and that $G_{\omega_0}$ contains an element $h$ of order $r$. From~\cite{GPS,GS}, the element $h$ has a cycle of length $r$ in its action on $\Lambda=\{1,\ldots,n\}$, and hence, relabeling the indexed set $\Lambda=\{1,\ldots,n\}$ if necessary, we may assume that
$$h=(1\,2\,3\,4\cdots \,r)h',$$
for some $h'\in \Sym(\{r+1,\ldots,n\})$.

Recall that from Bertrand's postulate, for each $\kappa\in\mathbb{N}$ with $\kappa>3$, there exists a prime $p$ with $\kappa<p<2\kappa-2$. Applying this result, we see that there exists a prime $p$ such that
\[
\begin{cases}
  \frac{q+5}{8}<p<2\cdot \frac{q+5}{8}-2=\frac{q-3}{4}&\textrm{when }q\equiv 3\pmod 8, q\notin \{11,19\},\\
  \frac{q+3}{8}<p<2\cdot \frac{q+3}{8}-2=\frac{q-5}{4}&\textrm{when }q\equiv 5\pmod 8, q\neq 13,\\
 2^{f-3}<p<2\cdot 2^{f-3}-2=2^{f-2}-2&\textrm{when }f\ge 5. 
\end{cases}
\]
We also set
\[
p:=
\begin{cases}
  3&\textrm{when }q=11,\\
   7&\textrm{when }q=19,\\
   5&\textrm{when }q=13,\\
   5&\textrm{when }2^f=2^4=16.\\
  \end{cases}
\]

Let $n_0$ be the minimal degree of a faithful permutation representation of $\PSL_2(q)$, and recall that $n_0=q+1$, except when $q\in\{5,7,9,11\}$ where the amended values are $5,7,6,11$ respectively. Since we are assuming that $q>8$ (and $q=2^f$ or $q\equiv 3,5\pmod 8$), we have $n_0=q+1$ except for $q=11$ where $n_0=11$.

Let $$\tau:=(1\,2\,\cdots\,p)$$ and $\omega_1:=\omega_0^\tau$, in particular, $G_{\omega_1}=G_{\omega_0}^\tau$. Let $u\in G_{\omega_0}^\tau\cap G_{\omega_0}=G_{\omega_1}\cap G_{\omega_0}$ and suppose that $u\ne 1$. Then, $u=v^\tau$ for some $v\in G_{\omega_0}\setminus\{1\}$. Now, $[v,\tau]=v^{-1}u\in G_{\omega_0}$ and moreover $[v,\tau]=(v^{-1}\tau^{-1}v)\tau$ is the product of two $p$-cycles of $\Sym(n)$. (Given $x\in \Sym(\Lambda)$, denote by $\mathrm{fix}_{\Lambda}(x)$ the number of fixed points of $x$ on $\Lambda$.) We have
\begin{eqnarray*}
  \frac{\mathrm{fix}_\Lambda([v,\tau])}{|\Lambda|}&\ge &\frac{n-2p}{n}=1-\frac{2p}{n}\ge 1-\frac{2p}{n_0}\ge 
  \begin{cases}
    1-\frac{6}{11}&\textrm{when }q=11,\\
     1-\frac{7}{10}&\textrm{when }q=19,\\
    1-\frac{5}{7}&\textrm{when }q=13,\\
     1-\frac{10}{17}&\textrm{when }q=16,\\
    1-\frac{q-3}{2(q+1)}&\textrm{when }q\equiv 3\pmod 8, q>19\\
    1-\frac{q-5}{2(q+1)}&\textrm{when }q\equiv 5\pmod 8, q>13\\
    1-\frac{q-8}{2(q+1)}&\textrm{when }q=2^f, q\ge 16.
    \end{cases}
\end{eqnarray*}
Now a very useful upper bound for $$\max\left\{\frac{\mathrm{fix}_\Lambda(x)}{|\Lambda|}\mid x\in G_{\omega_0}\setminus\{1\}\right\}$$ can be read off from~\cite[Theorem~$1$ and Table~$1$]{LiebeckSaxl}. (In using this information, observe that, when $q$ is odd, $q$ is not a square because $q\equiv 3,5\pmod 8$.) %, and that $n\ne q(q+1)/2$ because in the primitive permutation representation of degree $q(q+1)/2$  of $\mathrm{P}\Gamma \mathrm{L}_2(q)$ the inclusion of $\mathrm{P}\Gamma \mathrm{L}_2(q)$ is not maximal being contained in $\Sym(q+1)$ in its natural primitive action on the $2$-subsets of $\{1,\ldots,q+1\}$.)
 This upper bound implies that $[v,\tau]=1$, i.e.\ $\tau$ and $v$ commute (in other words, the lower bound for $\mathrm{fix}_\Lambda([v,\tau])/|\Lambda|$ obtained above is larger than the upper bound for the fixed-point-ratio of non-identity elements obtained in~\cite{LiebeckSaxl}).

As $\tau$ is a $p$-cycle, we obtain $v=\tau^s v'$, for some $v'\in \Sym(\{p+1,\ldots,n\})$ and some $s\in\mathbb{N}$. Suppose that $s=0$. Then $v=v'\in G_{\omega_0}$ fixes at least $p$ points of $\Lambda$. Therefore
\begin{eqnarray*}
  \frac{\mathrm{fix}_\Lambda(v)}{|\Lambda|}&\ge &\frac{p}{n}\ge \frac{p}{n_0}>
  \begin{cases}
    \frac{3}{11}&\textrm{when }q=11,\\
    \frac{7}{20}&\textrm{when }q=19,\\
    \frac{5}{14}&\textrm{when }q=13,\\
    \frac{5}{17}&\textrm{when }q=16,\\
    \frac{q+5}{8(q+1)}&\textrm{when }q\equiv 3\pmod 8, q>19\\
    \frac{q+3}{8(q+1)}&\textrm{when }q\equiv 5\pmod 8, q>13\\
    \frac{q}{8(q+1)}&\textrm{when }q=2^f, q>16.
    \end{cases}
\end{eqnarray*}
In each case, a careful case-by-case analysis of this lower bound with the upper bound for $\max_{x\ne 1}\{\mathrm{fix}_\Lambda(x)/|\Lambda|\}$ in~\cite[Theorem~$1$ and Table~$1$]{LiebeckSaxl} reveals that $v=1$, which is a contradiction. Therefore $s\neq 0$. By raising $v$ to a suitable power, we may assume that $v=\tau v'$, for some $v'\in \Sym(\{p+1,\ldots,n\})$. Now, the permutation $hv^{-1}$ belongs to $G_{\omega_0}$ and fixes at least $p-1$ points: indeed, $hv^{-1}$ fixes the points $1,2,\ldots,p-1$ of $\Lambda$. Arguing again as above, we see that $hv^{-1}$ fixes too many points to be an element of $G_{\omega_0}$, a contradiction. This contradiction has arisen from assuming that $G_{\omega_0}\cap G_{\omega_1}\ne 1$. Therefore
$$G_{\omega_0}\cap G_{\omega_1}=1.$$

Set $k:=h^\tau\in G_{\omega_0}^\tau=G_{\omega_1}$ and observe that $k=(2\,3\,\cdots p\,1\,p\!\!+\!\!1\cdots r)h'$, because $r>p$. Now set $z:=hk$ and observe that
$$z=(1\,3\,5\,\cdots p\!\!-\!\!2 \,p\,p\!\!+\!\!2\,\cdots r\,p\!\!+\!\!1\,p\!\!+\!\!3\,\cdots r\!\!-\!\!1\,2\,4\cdots p\!\!-\!\!1)h'^2.$$
Since $r$ is odd, $z$ has the same cycle structure as $h$ and $k$ and hence $z\in G_{\omega_2}$, for some $\omega_2\in \Omega$.

Observe that
$$h=zk^{-1}\in G_{\omega_0}\cap G_{\omega_2}G_{\omega_1}.$$
If $h\in G_{\omega_2}$, then
$$zh^{-2}=(r\,p\!\!-\!\!1\,r\!\!-\!\!1)\in G_{\omega_2},$$
but this is a contradiction because the only primitive groups containing a $3$-cycle are $\Alt(n)$ and $\Sym(n)$ by a celebrated theorem of Jordan. Therefore $h\notin G_{\omega_2}$. Now we immediately deduce from Lemma~\ref{l: auxiliary} that $G$ is not binary.

\smallskip

\noindent\textsc{Line 5 in Table~\ref{t: alt}.}

\smallskip

\noindent The proof follows from the proof of Line~1 and Line~2 and the comment following the statement of Proposition~\ref{p: alt}.
\end{proof}

Note that the exact arities of various actions of $\Alt(n)$ and $\Sym(n)$ have been worked out (see, for example, \cite{cherlin_martin}). These results imply Theorem~A in certain special cases.

\subsection{Proving Proposition~\ref{p: alt}}

Our job for the remainder of this section is to prove Proposition~\ref{p: alt}.
%It will be useful to deal first with $n$ small. The following lemma yields Proposition~\ref{p: alt} for $n\leq 27$; it can be proved with the invaluable help of the computer algebra system \texttt{magma}~\cite{magma}.
The case of $n=6$ is rather special and we dispose of this peculiarity by invoking the help of a computer-aided computation: we  construct all primitive faithful permutation representations of $\Alt(6)$, $\Sym(6)$, $\PGL_2(9)$, $M_{10}$ and $\mathrm{P}\Gamma \mathrm{L}_2(9)$ and we find in each case a beautiful subset, except for the cases listed in Table~\ref{t: alt}. Having this case dealt with,
from here on we will take $M$ to be a maximal subgroup of $G$, where $G=\Alt(n)$ or $\Sym(n)$. We are interested in the action of $G$ on $\Omega$, where $\Omega$ is the set of right cosets of $M$.

Recall that the maximal subgroups of $G$ are described by the O'Nan--Scott--Aschbacher theorem; we will use notation for this theorem consistent with \cite{LPS1, LPS2}. The description of $M$ is given in terms of its action on the set $\{1,\dots, n\}$: in its action on $\{1,\ldots,n\}$, the group $M$ can be either intransitive, or transitive but not primitive (that is, imprimitive), or primitive. For the proof of Proposition~\ref{p: alt}, we deal with each of these cases, in order, in the following sections.

\subsection{The group \texorpdfstring{$M$}{M} is intransitive}
If $M$ is intransitive on $\{1,\ldots,n\}$, then $\Omega$ can be identified with the set of $k$-subsets of $\{1,\dots, n\}$ for some $k$ such that $1\leq k< \frac{n}{2}$.

\begin{lem}\label{l: M intrans 1}
Suppose that $2<k<n/2$, and that $\Omega$ is the set of $k$-subsets of $\{1,\ldots,n\}$. Then 
\begin{align*}
\Delta:=\{&\{1,2,3,8,9,\ldots,k+4\},\{3,4,5,8,9,\ldots,k+4\},\{1,5,6,8,9,\ldots,k+4\},\\
&\{1,4,7,8,9,\ldots,k+4\},\{3,6,7,8,9,\ldots,k+4\},\{2,5,7,8,9,\ldots,k+4\},\\
&\{2,4,6,8,9,\ldots,k+4\}\}
\end{align*}
is $G$-beautiful.
\end{lem}
\begin{proof}
 We start by observing that $\Delta$ consists of seven sets of size $k$ and that the seven sets 
\begin{equation}\label{eq:0}
\{1,2,3\},\{3,4,5\},\{1,5,6\},\{1,4,7\},\{3,6,7\},\{2,5,7\},\{2,4,6\}
\end{equation} are the lines of a Fano plane with point set $\{1,2,3,4,5,6,7\}$. 

Consider $$X:=\bigcap_{\delta\in\Delta}\delta=\{8,9,\ldots,k+4\},\quad Y:=\{1,\ldots,n\}\setminus\bigcup_{\delta\in \Delta}\delta=\{k+5,k+6,\ldots,n\}$$ and observe that, for every $\sigma\in G_{\Delta}$, $\sigma$ fixes set-wise $X$ and  $Y$. From this it immediately follows that $$G_{\Delta}\leq \Sym(\{1,\ldots,7\})\times \Sym(\{8,\ldots,k+4\})\times \Sym(\{k+5,\ldots,n\})$$ and $G^\Delta\cong \PGL_3(2)$ in its natural $2$-transitive action on the lines of the Fano plane.
\end{proof}

\begin{lem}\label{l: M intrans 2}
Suppose that $\Omega$ is the set of $2$-subsets of $\{1,\dots, n\}$. Then no subset of $\Omega$ is $G$-beautiful.
\end{lem}
\begin{proof}
We argue by contradiction and we assume that  $\Delta$ is  a beautiful subset of $\Omega$: in particular $\Delta\neq \emptyset$ and $\Delta\neq \Omega$. We can think of $\Delta$ as the edge set of a graph with vertex set $\{1,\ldots,n\}$: let $\Gamma$ be this graph.  Then $G_{\Delta}=\Aut(\Gamma)$ and, as $\Delta$ is a beautiful subset,  $\Aut(\Gamma)$ acts $2$-transitively on the edges of $\Gamma$. (As we noticed above, $\Gamma$ is neither empty nor complete.) 

If $\Gamma$ has two distinct edges with no vertex in common, then, by the $2$-transitivity of $\Aut(\Gamma)$ on edges, we infer that any two distinct edges of $\Gamma$ have no vertex in common. Therefore $\Gamma$ is a union of parallel edges and it easily follows that $G^\Delta\ge \Alt(\Delta)$, a contradiction.

Suppose that $\Gamma$ has two distinct edges with one vertex in common. Then, by the $2$-transitivity of $\Aut(\Gamma)$ on edges, we infer that any two distinct edges of $\Gamma$ have a vertex in common. Recall that we may assume that $n\ge 5$. A moment's thought (or by applying the Erd\H{o}s-Ko-Rado theorem~\cite{MR0140419}) gives that all edges of $\Gamma$ are adjacent to a fixed element of $\{1,\ldots,n\}$. Therefore $G^\Delta\geq \Alt(\Delta)$, a contradiction.
\end{proof}

\subsection{The group \texorpdfstring{$M$}{M} is transitive but not primitive} Here we suppose that $M$ is transitive but not primitive. In this case, $\Omega$ can be identified with the set of partitions of $\{1,\ldots, n\}$ into $\ell$ subsets, each of size $k$; in particular, $n=k\ell$ with $k,\ell>1$. We adopt this notation for the next three results.

\begin{lem}\label{l: M imprim 1}
 If $n=8$ and $k=2$, then the action of $G$ admits a beautiful subset of size $7$. If $n\geq 10$ and $k=2$, then the action of $G$ admits a beautiful subset of size $6$. 
\end{lem}

Note that when $n=6$, the action on partitions with $k=2$ is isomorphic to the action of $G$ on $2$-subsets of $\{1,\dots, n\}$. Hence the action admits no beautiful subset, see Lemma~\ref{l: M intrans 2}.

\begin{proof}
It is an easy computation with \texttt{magma} to confirm that $G$ has a beautiful subset of size $7$ when $n=8$.   Assume then that $n/k\geq 5$. 

Consider the Petersen graph $\mathcal{P}$ with the labeling as in Figure~\ref{figure2}. We recall that the automorphism group of $\mathcal{P}$ is isomorphic to $\Sym(5)$, where the action of $\Aut(\mathcal{P})$ on the vertices of $\mathcal{P}$ is permutation isomorphic to the natural action of $\Sym(5)$ on the $2$-subsets of $\{1,2,3,4,5\}$. Observe also that $\mathcal{P}$ has six complete matchings and that the action $\Aut(\mathcal{P})$ on these six complete matchings is permutation isomorphic to the natural $2$-transitive action of $\PGL_2(5)\cong\Sym(5)$ on the six points of the projective line. These six complete matchings are:
\begin{align}\label{eq:1}
\delta_1:=\{\{1,6\},\{2,7\},\{3,8\},\{4,9\},\{5,10\}\},\\\nonumber
\delta_2:=\{\{1,2\},\{3,8\},\{4,5\},\{6,9\},\{7,10\}\},\\\nonumber
\delta_3:=\{\{1,5\},\{2,3\},\{4,9\},\{6,8\},\{7,10\}\},\\\nonumber
\delta_4:=\{\{1,6\},\{2,3\},\{4,5\},\{7,9\},\{8,10\}\},\\\nonumber
\delta_5:=\{\{1,2\},\{3,4\},\{5,10\},\{6,8\},\{7,9\}\},\\\nonumber
\delta_6:=\{\{1,5\},\{2,7\},\{3,4\},\{6,9\},\{8,10\}\}.\nonumber
\end{align}

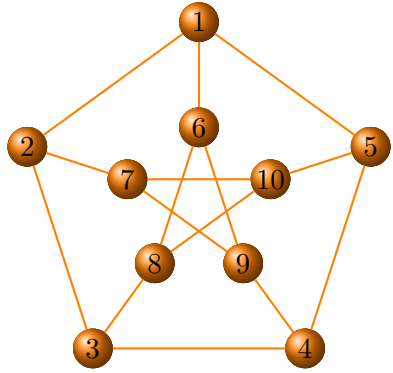
\begin{figure}[!ht]
\begin{tikzpicture}[rotate=90]
  \GraphInit[vstyle=Art]
  \SetVertexNoLabel
  \SetUpVertex[MinSize=15pt]
  \grPetersen[RA=2.4,RB=1]
  \AssignVertexLabel{a}{1,2,3,4,5} \AssignVertexLabel{b}{6,7,8,9,10}
\end{tikzpicture}
\caption{Labeling for the Petersen graph $\mathcal{P}$}\label{figure2}
\end{figure}

An easy computation shows that the only elements of $\Sym(10)$ fixing set-wise $\{\delta_1,\delta_2,\delta_3,\delta_4,\delta_5,\delta_6\}$ are the elements of $\Aut(\mathcal{P})$.

Using the complete matchings of $\mathcal{P}$, we construct a subset $\Delta$ of $\Omega$ of size six.
 Let $X_1,\ldots,X_{n/k-5}$ be a partition of $\{11,\ldots,n\}$ with $|X_i|=k$ for every $i\in \{1,\ldots,n/k-5\}$. Now define
\begin{align*}
\Delta:=\{&\{
\{1,6\},\{2,7\},\{3,8\},\{4,9\},\{5,10\},X_1,\ldots,X_{n/k-5}\},\\
&\{\{1,2\},\{3,8\},\{4,5\},\{6,9\},\{7,10\},X_1,\ldots,X_{n/k-5}\},\\
&\{\{1,5\},\{2,3\},\{4,9\},\{6,8\},\{7,10\},X_1,\ldots,X_{n/k-5}\},\\
&\{\{1,6\},\{2,3\},\{4,5\},\{7,9\},\{8,10\},X_1,\ldots,X_{n/k-5}\},\\
&\{\{1,2\},\{3,4\},\{5,10\},\{6,8\},\{7,9\},X_1,\ldots,X_{n/k-5}\},\\
&\{\{1,5\},\{2,7\},\{3,4\},\{6,9\},\{8,10\},X_1,\ldots,X_{n/k-5}\}   \}.
\end{align*}
From the discussion above, $G^{\Delta}$ is permutation isomorphic to $\PGL_2(5)$ in its natural $2$-transitive action of degree $6$ and hence $\Delta$ is a beautiful subset.
\end{proof}

\begin{lem}\label{l: M imprim 2}
 If $k\geq 4$, then the action of $G$ admits a beautiful subset of size $7$.
\end{lem}
\begin{proof}
Let  $X_1,\ldots,X_{n/k-2}$ be a partition of $\{2k+1,2k+2,\ldots,n\}$ with $|X_i|=k$, for every $i\in \{1,\ldots,n/k-2\}$. Using the seven lines of the Fano plane in Eq.~\eqref{eq:0}, we construct a subset $\Delta$ of $\Omega$ of size $7$:
\begin{align*}
\Delta:=\{&\delta_1:=\{
\{1,2,3\}\cup \{8,\ldots,k+4\},\{4,5,6,7\}\cup \{k+5,k+6,\ldots,2k\},X_1,\ldots,X_{n/k-2}\},\\
&\delta_2:=\{\{3,4,5\}\cup \{8,\ldots,k+4\},\{1,2,6,7\}\cup \{k+5,k+6,\ldots,2k\},X_1,\ldots,X_{n/k-2}\},\\
&\delta_3:=\{\{1,5,6\}\cup \{8,\ldots,k+4\},\{2,3,4,7\}\cup \{k+5,k+6,\ldots,2k\},X_1,\ldots,X_{n/k-2}\},\\
&\delta_4:=\{\{1,4,7\}\cup \{8,\ldots,k+4\},\{2,3,5,6\}\cup \{k+5,k+6,\ldots,2k\},X_1,\ldots,X_{n/k-2}\},\\
&\delta_5:=\{\{3,6,7\}\cup \{8,\ldots,k+4\},\{1,2,4,5\}\cup \{k+5,k+6,\ldots,2k\},X_1,\ldots,X_{n/k-2}\},\\
&\delta_6:=\{\{2,5,7\}\cup \{8,\ldots,k+4\},\{1,3,4,6\}\cup \{k+5,k+6,\ldots,2k\},X_1,\ldots,X_{n/k-2}\},\\
&\delta_7:=\{\{2,4,6\}\cup \{8,\ldots,k+4\},\{1,3,5,7\}\cup \{k+5,k+6,\ldots,2k\},X_1,\ldots,X_{n/k-2}\}
\}.
\end{align*}
(Observe that the hypothesis $k\ge 4$ guarantees that $\Delta$ consists of partitions of $\{1,\ldots,n\}$ into $n/k$ parts of size exactly $k$.)

By construction, we immediately see that $G^\Delta\geq \PGL_3(2)$ with its natural action on the seven lines of the Fano plane. Suppose that $G^\Delta\geq \Alt(\Delta)$. In particular, $G_{\Delta}$ contains a permutation $\sigma$ fixing $\delta_1,\delta_2,\delta_3,\delta_4$ and permuting cyclically $\delta_5,\delta_6,\delta_7$. 

Observe that $$\{X_1,\ldots,X_{n/k-2}\}=\bigcap_{\delta\in \Delta}\delta.$$
As $\sigma$ fixes set-wise $\Delta$, we deduce that $X_i^\sigma\in \{X_1,\ldots,X_{n/k-2}\}$, for every $i\in \{1,\ldots,n/k-2\}$. In particular, without loss of generality we may assume that $\sigma$ fixes point-wise $\{2k+1,\ldots,n\}$. Now, let $i,j\in \{1,2,3,4\}$ with $i\ne j$. If $\sigma$ interchanges the first two parts of $\delta_i$, then $\sigma$ does not fix $\delta_j$, contradicting the fact that $\sigma$ fixes the partitions $\delta_i$ and $\delta_j$. This shows that $\sigma$ fixes set-wise the first two parts of the partitions $\delta_1,\delta_2,\delta_3,\delta_4$. In particular, without loss of generality we may assume that $\sigma$ fixes point-wise $\{8,\ldots,k+4\}$ and $\{k+5,\ldots,2k\}$. Therefore $\sigma$ fixes point-wise $\{8,9,\ldots,n\}$. As $\sigma\in G_{\Delta}$, we now deduce that the seven lines of the Fano plane are  $\sigma$-invariant and hence $\sigma$ is an automorphism of the Fano plane. 

Since $\sigma$ fixes $\delta_1,\delta_2,\delta_3,\delta_4$, we get that $\sigma$ fixes four lines of the Fano plane. A computation yields that the only automorphism of the Fano plane fixing four lines is the identity, and hence $\sigma=1$, a contradiction. Therefore $G^{\Delta}\ngeq \Alt(\Delta)$.
\end{proof}

\begin{lem}\label{l: M imprim 3}
 If $k=3$, then the action of $G$ admits a beautiful subset of size $10$.
\end{lem}
\begin{proof}
Recall that when $n=6$, there are $10$ partitions of $\{1,2,3,4,5,6\}$ into two parts of size three, the action of $G$ on this set $\Omega$ is $2$-transitive of degree $10$ and $G\ngeq \Alt(\Omega)$; hence $\Omega$ is beautiful. In fact, the action of $\Alt(6)\cong \PSL_2(9)$ on $\Omega$ is permutation isomorphic to the natural action of $\PSL_2(9)$ on the points of the projective line. 

Let $X_1,\ldots,X_{n/k-2}$ be a partition of $\{7,8,\ldots,n\}$ with $|X_i|=3$, for every $i\in \{1,\ldots,n/3-2\}$. Using the ten $3$-uniform partitions of $\{1,\ldots,6\}$ we construct a subset $\Delta$ of $\Omega$ of size $10$:
\begin{align*}
\Delta:=\{&\{
\{1,2,3\},\{4,5,6\},X_1,\ldots,X_{n/3-2}\},\,\{\{1,2,4\},\{3,5,6\},X_1,\ldots,X_{n/3-2}\},\\
&\{\{1,2,5\},\{3,4,6\},X_1,\ldots,X_{n/3-2}\},\,\{\{1,2,6\},\{3,4,5\},X_1,\ldots,X_{n/3-2}\},\\
&\{\{1,3,4\},\{2,5,6\},X_1,\ldots,X_{n/3-2}\},\,\{\{1,3,5\},\{2,4,6\},X_1,\ldots,X_{n/3-2}\},\\
&\{\{1,3,6\},\{2,4,5\},X_1,\ldots,X_{n/3-2}\},\,\{\{1,4,5\},\{2,3,6\},X_1,\ldots,X_{n/3-2}\},\\
&\{\{1,4,6\},\{2,3,5\},X_1,\ldots,X_{n/3-2}\},\,\{\{1,5,6\},\{2,3,4\},X_1,\ldots,X_{n/3-2}\}
\}.
\end{align*}
It is readily seen, arguing as in the proof of Lemma~\ref{l: M imprim 2}, that $G^\Delta$ is permutation isomorphic to  $\Sym(6)$ or $\Alt(6)$ in its $2$-transitive action of degree $10$, and hence $\Delta$ is beautiful.
\end{proof}

\subsection{The group \texorpdfstring{$M$}{M} is primitive. }Here we consider those maximal subgroups $M$ that act primitively on $\{1,\dots, n\}$. The O'Nan--Scott--Aschbacher theorem asserts (in the terminology in~\cite{LPS1,LPS2}) that we are in one of the following situations.
\begin{description}
\item[(AS)] $M$ is \emph{almost simple}; or
\item[(Aff)] $M$ is \emph{affine}: $n=r^\ell$ for some prime $r$ and positive integer $\ell$, and $M=\AGL_\ell(r)$ when $G=\Sym(n)$ or $M=\AGL_\ell(r)\cap \Alt(n)$ when $G=\Alt(n)$; or
\item[(Diag)] $M$ is a primitive group of \emph{diagonal} type; or
\item[(Prod)] $n=m^\iota$ where $m,\iota\in\mathbb{N}$ with $m\ge 5$ and $\iota\ge 2$, and $M=\Sym(m)\wr\Sym(\iota)$ when $G=\Sym(n)$ or $M=(\Sym(m)\wr\Sym(\iota))\cap \Alt(n)$ when $G=\Alt(n)$; this is the \emph{product} action of $M$. 
\end{description}

In order to deal with these four cases, we need a little notation. We write $\Lambda:=\{1,\dots, n\}$ and we consider the actions of $G$ and $M$ on $\Lambda$: as before, for $x\in G$, we write 
\[
\fix_\Lambda(x) := |\{\lambda \in \Lambda \mid \lambda^x = \lambda\}|.
\]
On the other hand, $G$ also acts on $\Omega$, the set of right cosets of $M$ in $G$, and we write $\omega_0$ for the coset $M$, considered as an element of $\Omega$.
Now the following lemma will be used repeatedly:

\begin{lem}\label{l: beaut}
Let $k$ and $t$ be positive integers with $t\ge 3$, and suppose that the following conditions hold:
\begin{enumerate}
 \item  $M$ contains an element $h$ of order $t-1$; 
 \item $G\setminus M$ contains an element of $g$ of order $t$ with $\fix_\Lambda(g)=n-k$; and
 \item $K:=\langle g,h\rangle=\langle g\rangle \rtimes \langle h \rangle$ is a Frobenius group with Frobenius kernel $\langle g\rangle$.
\end{enumerate}
Then one of the following holds:
 \begin{enumerate}
  \item $M$ contains a non-trivial element $f$ of order divisible by $3$ such that $\fix_\Lambda(f)\geq n-2k$;
  \item $\Delta:=\omega_0^K$ is a $G$-beautiful subset of $\Omega$ of size $t$.
 \end{enumerate}
\end{lem}
\begin{proof}
Since $K$ is Frobenius, the group $\langle h\rangle$ acts by conjugation fixed-point-freely on the group $\langle g\rangle$. Since $h\in M$ and $g\not\in M$, we conclude that $K_{\omega_0}=K\cap G_{\omega_0}=K\cap M=\langle h\rangle$. Thus the action of $K$ on $\Delta$ is permutation isomorphic to the $2$-transitive action of $K$ on the right cosets of $\langle h\rangle$; in particular $K$ acts faithfully and $2$-transitively on $\Delta$.

Next observe that $K$ contains an odd and an even cycle on $\Delta$, and so $K$ (viewed as a permutation group on $\Delta$) cannot be contained in $\Alt(\Delta)$. Thus, either $\Delta$ is a $G$-beautiful subset of $\Omega$, or else $G^\Delta= \Sym(\Delta)$. 

Suppose, then, that $G^\Delta=\Sym(\Delta)$. Observe that $g$ acts as a $t$-cycle on $\Delta$. Thus, writing $\Delta=\{\omega_0,\omega_1,\dots, \omega_{t-1}\}$, we can assume that $g$ acts as the $t$-cycle $(\omega_0\,\omega_1\,\dots\, \omega_{t-1})$ on $\Delta$. Let $x$ be an element of  $G_{\Delta}$ that acts as the transposition $(\omega_0\,\omega_1)$ on $\Delta$.

Consider $y:=(x^{-1})^g$ which acts as the transposition $(\omega_1\,\omega_2)$ on $\Delta$. Since both $x$ and $y$ fix $\omega_0$, they are both elements of $M$. Now the product $xy$ acts as the $3$-cycle $(\omega_0\,\omega_2\,\omega_1)$ on $\Delta$, hence is a non-trivial element of $M$ of order divisible by $3$.

What is more, by supposition, the element $g$ fixes $n-k$ points of $\Lambda$, and hence so does its conjugate $xg^{-1}x^{-1}$. Therefore the product $xg^{-1}x^{-1}\cdot g=xy$ fixes at least $n-2k$ points of $\Lambda$. 
\end{proof}

Lemma~\ref{l: beaut} will be used in conjunction with known upper bounds for the number of fixed points of non-trivial elements in $M$ in its action on $\Lambda$. This information can be found in Table~\ref{t: fixed points}. 
(By {\rm\bf (AS*)} we mean that $M$ is an almost simple group, and if the socle of $M$ is isomorphic  to $\Alt(m)$ for some $m$, then the action of $M$ on $\Lambda$ is not permutation isomorphic to the action on the $k$-element subsets of $\{1,\dots,m\}$ for some $k\in \{1,\ldots,m-1\}$.) 
For further details concerning the notation used, we refer to the proofs of the ensuing lemmas, each of which deals with one of the four families.

 \begin{center}
\begin{table}[!ht]
\begin{tabular}{ccc}
\toprule[1.5pt]
Type of $M$ & Upper bound & Reference \\
\midrule[1.5pt]
{\rm\bf (AS)} & $n-2(\sqrt{n}-1)$ & \cite[Corollary~3]{LiebeckSaxl}\footnote{\cite[Corollary~3]{LiebeckSaxl} is a slight strengthening of~\cite{BABAI_1}.} \\
{\rm\bf (AS*)} & $4n/7$ & \cite[Theorem~1 and Corollary~$1$]{GM}\footnote{Theorem~$1$ in~\cite{GM} gives the best known upper bound for the number of fixed points of a non-identity element of a primitive group. We apply this theorem only for almost simple primitive groups $M$, where the socle of $M$ is not an alternating group $\Alt(m)$ in its natural action on $k$-subsets. In particular, under these conditions Part~2 of Theorem~$1$ in~\cite{GM} does not apply, and hence we deduce the bound given in Table~\ref{t: fixed points}.} \\
{\rm\bf (Aff)} & $n/r$,\, where $r$ is prime and $n=r^\ell$ & Elementary\\
{\rm\bf (Diag)} & $4n/15$ & \cite[Section~$5$]{GPS} \\
{\rm\bf (Prod)} & $(|\Gamma|-2)|\Gamma|^{\iota-1}$ where $\Lambda=\Gamma^\iota$ & \cite[Lemma~$6.13$]{GPS} \\
\bottomrule[1.5pt]
\end{tabular}
\caption{Upper bounds for $\max\{\fix_\Lambda(x)\mid x\in M\setminus\{1\}\}$}\label{t: fixed points}
\end{table}
\end{center}

\begin{lem}\label{l: M prod}
 If $M$ is of type {\rm\bf (Prod)}, then the action of $G$ on the cosets of $M$ admits a beautiful subset of size $5$.
\end{lem}
\begin{proof}
In this case $\Lambda$ can be identified with an $M$-invariant Cartesian decomposition $\Gamma^\iota$, where $\Gamma$ is a set of size $m\ge 5$, and $n=m^\iota$, where $\iota\in\mathbb{N}$ and $\iota\ge 2$. The action of $M$ on $\Gamma^\iota$ is the natural primitive product action on the Cartesian decomposition.

We write the elements of $M$ in the form $(x_1,\ldots,x_\iota)\sigma$, where $x_1,\ldots,x_\iota\in \Sym(\Gamma)$ and $\sigma\in\Sym(\iota)$. In particular, given $(\gamma_1,\ldots,\gamma_\iota)\in\Gamma^\iota$ and $m=(x_1,\ldots,x_\iota)\sigma\in M$, we get that $$(\gamma_1,\gamma_2,\ldots,\gamma_\iota)^m:=(\gamma_{1^{\sigma^{-1}}}^{x_{1^{\sigma^{-1}}}},\gamma_{2^{\sigma^{-1}}}^{x_{2^{\sigma^{-1}}}},\ldots,\gamma_{\iota^{\sigma^{-1}}}^{x_{\iota^{\sigma^{-1}}}}).$$

Fix $\gamma_0,\gamma_1,\gamma_2,\gamma_3,\gamma_4$ five distinct elements of $\Gamma$.
Consider the element $$h:=((\gamma_1,\gamma_2,\gamma_3,\gamma_4),(\gamma_1,\gamma_2),1,1,\ldots,1)$$ of $M$. Clearly, $h$ fixes a point of $\Gamma^\iota$, for instance $h$ fixes the point $\delta_0:=(\gamma_0,\gamma_0,\ldots,\gamma_0)\in\Gamma^{\iota}$. Moreover, $h$ acts as an even permutation on $\Gamma^\iota$ and hence $h\in \Alt(n)$. Observe that the $\langle h\rangle$-orbit containing $(\gamma_1,\gamma_1,\gamma_0,\ldots,\gamma_0)$ has length $4$. Namely, it consists of 
\begin{align*}
&\delta_1:=(\gamma_1,\gamma_1,\gamma_0,\ldots,\gamma_0),&&\delta_2:=\delta_1^h=(\gamma_2,\gamma_2,\gamma_0,\ldots,\gamma_0),\\
&\delta_3:=\delta_2^h=(\gamma_3,\gamma_1,\gamma_0,\ldots,\gamma_0),&&\delta_4:=\delta_3^h=(\gamma_4,\gamma_2,\gamma_0,\ldots,\gamma_0).
\end{align*}
Consider the $5$-cycle $g:=(\delta_0\,\delta_1\,\delta_2\,\delta_4\,\delta_3)$ of $\Alt(\Gamma^\iota)=\Alt(n)$. Clearly $g^h=g^2$ and so $\langle h\rangle$ normalizes $\langle g\rangle$ and $\langle g,h\rangle$ is a Frobenius group of order $20$. 

Our aim is to apply Lemma~\ref{l: beaut} to the elements $g$ and $h$ with $t:=5$ and $k:=5$. By construction $h$ is an element of $M$ of order $4=t-1$, $g$ has order $5=t$ and $\mathrm{fix}_\Lambda(g)=n-5=n-k$. Our remaining job is to verify that $g$ is not an element of $M$.

We argue by contradiction and we assume that $g\in M$. Table~\ref{t: fixed points} implies that $g$ fixes at most $(|\Gamma|-2)|\Gamma|^{\iota-1}$ points of $\Gamma^\iota\cong\{1,\ldots,n\}=\Lambda$. Hence $|\Gamma|^\iota-5=n-5=\mathrm{fix}_\Lambda(g)\leq (|\Gamma|-2)|\Gamma|^{\iota-1}$, a contradiction. 

Now Lemma~\ref{l: beaut} implies that either $\omega_0^K$ is a beautiful subset of size $5$ (and we are done), or else $M$ contains a non-trivial element $f$ of order divisible by $3$ such that $\fix_\Lambda(f)\geq n-10$. Again, Table~\ref{t: fixed points} implies that $\fix_\Lambda(f) \leq (|\Gamma|-2)|\Gamma|^{\iota-1}$. This yields $\iota=2$, $|\Gamma|=5$, $n=5^2=25$ and $f$ fixes $15$ points of $\Gamma^\iota$. Now, a computation shows that each permutation of $\Sym(5)\wr \Sym(2)$ fixing $15$ points of $\Gamma^\iota$ has order $2$, contradicting the fact that $f$ has order divisible by $3$.
 \end{proof}

Our next job is to set up some notation that will be useful for dealing with the case that $M$ is primitive of diagonal type. 
We start by recalling the structure of the finite primitive groups 
of diagonal type. 

Let $\ell\geq 1$ and let $T$ be a non-abelian simple group. Consider
the group $N:=T^{\ell+1}$ and $D:=\{(t,\ldots,t)\in N\mid t\in T\}$, a
diagonal subgroup of $N$. Let $\Lambda$ be the set of 
right cosets of $D$ in $N$. Then $|\Lambda|=|T|^\ell$. Moreover we
may identify each element $\lambda\in \Lambda$ with an element
of $T^\ell$ as follows: the right coset
$\lambda=D(\alpha_0,\alpha_1,\ldots,\alpha_\ell)$
contains a unique element whose first coordinate is $1$, namely, the
element
$(1,\alpha_0^{-1}\alpha_1,\ldots,\alpha_0^{-1}\alpha_\ell)$. We
choose this
distinguished coset representative and we denote the element 
$D(1,\alpha_1,\ldots,\alpha_\ell)$ of $\Lambda$
simply by
$$
[1,\alpha_1,\ldots,\alpha_\ell].
$$
Now the element $\varphi$ of $\Aut(T)$ acts on $\Lambda$ by
$$
[1,\alpha_1,\ldots,\alpha_\ell]^\varphi=[1,\alpha_1^\varphi,\ldots,\alpha_\ell^\varphi].
$$
Note that this action is well-defined because $D$ is
$\Aut(T)$-invariant.
Next, the element $(t_0,\ldots,t_\ell)$ of $N$ acts on $\Lambda$ by
$$
[1,\alpha_1,\ldots,\alpha_\ell]^{(t_0,\ldots,t_\ell)}=
[t_0,\alpha_1t_1,\ldots,\alpha_\ell
  t_\ell]=[1,t_0^{-1}\alpha_1t_1,\ldots,t_0^{-1}\alpha_\ell t_\ell].
$$
Observe that the action induced by  $(t,\ldots,t)\in N$ on $\Lambda$ is
the same as the action induced by the inner automorphism 
corresponding to conjugation by $t$.
Finally, the element $\sigma$ in $\Sym(\{0,\ldots,\ell\})$  acts
on $\Lambda$ simply by permuting the coordinates. Note that this action
is well-defined because $D$ is $\Sym(\ell+1)$-invariant.

The set of all permutations we described generates a group $W$ isomorphic
to $T^{\ell+1}\cdot (
\Out(T)\times \Sym(\ell+1))$. 

In our case, we may identify the set $\{1,\ldots,n\}$ with $\Lambda=N/D$. Moreover, $M=W$ when $G=\Sym(n)$ and $M=W\cap \Alt(n)$ when $G=\Alt(n)$.

 Write 
$$
R=\{(t_0,t_1,\ldots,t_\ell)\in N\mid t_0=1\}.
$$ 
Clearly, $R$ is a normal subgroup of $N$ acting
regularly on $\Lambda$. Since the stabilizer in $W$ of the point
$[1,\ldots,1]$ is $\Sym(\ell+1)\times \Aut(T)$, we obtain
that 
$$
W=(\Sym(\ell+1)\times \Aut(T))R.
$$ 
Moreover,  every element $x\in W$
can be written uniquely as $x=\sigma\varphi r$, with $\sigma\in \Sym(\ell+1)$,
$\varphi\in\Aut(T)$ and $r\in R$. 

The next lemma deals with the groups of diagonal type; we use the notation just established.

\begin{lem}\label{l: M diag}
 If $M$ is of type {\rm\bf (Diag)}, then the action of $G$ on the cosets of $M$ admits a beautiful subset of size $5$.
\end{lem}
\begin{proof}
We start with a claim:

\noindent\textsc{Claim. }The group $M$ contains an element $h$ of order $4$.

\noindent\textsc{The claim implies the result.} Suppose that the claim is true. Now write $\Lambda:=\{1,\ldots,n\}$, and consider the action of $M$ on $\Lambda$ that induces the embedding of $M$ in $G$. Under this action, let $n_1$ be the number of fixed points of $h$, $n_2$ the number of cycles of length $2$ of $h$ and $n_4$ the number of cycles of length $4$ of $h$. Clearly, $n=n_1+2n_2+4n_4$ and $h^2$ fixes $n_1+2n_2=n-4n_4$ points of $\Lambda$.  From~\cite[Lemmas~$5.3$ and~$5.4$]{GPS}, we see that $h^2$ fixes at most $4n/15$ points of $\Lambda=\{1,\ldots,n\}$ (see also Table~\ref{t: fixed points}). Therefore $n-4n_4\le 4n/15$ and $n_4\ge 11n/60$. Since $n=|\Lambda|=|T^{\iota}|\ge 60$, we get $n_4\ge 11$ and hence $h$ has at least $11$ cycles of length $4$. For our argument we only need five cycles of $h$ of length $4$. Say that, in its cycle decomposition, 
$$h=(\lambda_1\,\lambda_2\,\lambda_3\,\lambda_4)
(\lambda_5\,\lambda_6\,\lambda_7\,\lambda_8)
(\lambda_9\,\lambda_{10}\,\lambda_{11}\,\lambda_{12})
(\lambda_{13}\,\lambda_{14}\,\lambda_{15}\,\lambda_{16})
(\lambda_{17}\,\lambda_{18}\,\lambda_{19}\,\lambda_{20})\cdots.$$
Consider the element $g\in\Sym(\Lambda)$ having cycle structure
$$g:=
(\lambda_1\,\lambda_5\,\lambda_9\,\lambda_{13}\,\lambda_{17})
(\lambda_2\,\lambda_{10}\,\lambda_{18}\,\lambda_6\,\lambda_{14})
(\lambda_3\,\lambda_{19}\,\lambda_{15}\,\lambda_{11}\,\lambda_7)
(\lambda_4\,\lambda_{16}\,\lambda_8\,\lambda_{20}\,\lambda_{12}).$$
A computation gives $g^h=g^3$ and so $\langle h\rangle$ normalizes $\langle g\rangle$ and $\langle g,h\rangle$ is a Frobenius group of order $20$. Set $K:=\langle g,h\rangle$.

We wish to apply Lemma~\ref{l: beaut} to the elements $g$ and $h$ with $t:=5$ and $k:=20$. By construction $h$ is an element of $M$ of order $4=t-1$, $g$ has order $5=t$ and $\mathrm{fix}_\Lambda(g)=n-20=n-k$. Our remaining job is to verify that $g\not\in M$. We argue by contradiction and we assume that $g\in M$. Table~\ref{t: fixed points} implies that $n-20=\mathrm{fix}_\Lambda(g)\leq 4n/15$, and so $n\leq 27$. However this contradicts $n=|\Lambda|=|T^{\ell}|\ge 60$. 

Now Lemma~\ref{l: beaut} implies that either $\omega_0^K$ is a beautiful subset of size $5$ (and we are done), or else $M$ contains a non-trivial element $f$ such that $\fix_\Lambda(f)\geq n-40$. In the latter case, Table~\ref{t: fixed points} implies, again, that $n-40\leq 4n/15$, and so $n\leq 54$ which, again, contradicts the fact that $n\geq 60$.

\noindent\textsc{Proof of the claim.}  From the description of $M$, we see that $M$ contains an element of order $4$ whenever $T$ does. In particular, we may assume that $T$ has no element of order $4$ and so the Sylow $2$-subgroups of the non-abelian simple group $T$ are elementary abelian. Now, by a celebrated theorem of Walter~\cite{walter}, we get that $T$ is one of the following groups: $\PSL_2(2^f)$ for some $f\in \mathbb{N}$ with $f\ge 2$, $\PSL_2(q)$ with $q\equiv 3,5\pmod 8$, $J_1$, $\Ree(3^{2f+1})$ with $f\ge 1$.

Let $t$ be an involution of $T$ and consider the element  $m_1:=(1,t,1,\ldots,1)$ of $R$. Clearly, $m_1$ has order $2$. Consider then $\sigma:=(0\,1)\in \Sym(\ell+1)\leq W$. Now, for each $[1,\alpha_1,\alpha_2,\ldots,\alpha_\ell]\in \Lambda$, we have 
\[
[1,\alpha_1,\alpha_2,\ldots,\alpha_\ell]^\sigma=[1, \alpha_1^{-1},\alpha_1^{-1}\alpha_2,\ldots,\alpha_1^{-1}\alpha_\ell].
\]
If $\ell+1\geq 3$ (that is, $\ell \ge 2$), then this element is fixed by $\sigma$ if and only if $\alpha_1=1$. Thus $\sigma$ (viewed as a permutation in $\Sym(\Lambda)$) fixes $|T|^{\ell-1}$ points. Thus the number of cycles of length $2$ of $\sigma$ is 
\[
\frac{|T|^\ell-|T|^{\ell-1}}{2}=\frac{|T|-1}{2}|T|^{\ell-1}.
\]
Since $|T|$ is divisible by $4$, this number is clearly even, and hence $\sigma\in W\cap \Alt(\Lambda)\le M$.

If $\ell=1$, then the element $[1,\alpha_1,\ldots,\alpha_\ell]=[1,\alpha_1]$ above is fixed by $\sigma$ if and only if $\alpha_1^2=1$. In particular, $\sigma$ (viewed as a permutation in $\Sym(\Lambda)$) fixes $\iota$ points, where $\iota=|\{\alpha\in T\mid \alpha^2=1\}|$. Thus the number of cycles of length $2$ of $\sigma$ is $$\frac{|\Lambda|-\mathrm{fix}_\Lambda(\sigma)}{2}=\frac{|T|-\iota}{2}.$$
A direct computation in $\PSL_2(2^f)$, $\PSL_2(q)$, $J_1$ and $\Ree(3^{2f+1})$ reveals that $(|T|-\iota)/2$ is even. In particular, note that
\[\iota=
 \begin{cases}
   2^{2f}&\textrm{when }T=\PSL_2(2^f),\\
\frac{q^2-q+2}{2}&\textrm{when }T=\PSL_2(q),\,q\equiv 3\pmod 8,\\
   \frac{q^2+q+2}{2}&\textrm{when }T=\PSL_2(q),\,q\equiv 5\pmod 8,\\
   1464&\textrm{when }T=J_1,\\
   3^{8f+4}-3^{6f+3}+3^{4f+2}+1&\textrm{when }T=\Ree(3^{2f+1}).
 \end{cases}
 \]
Hence again $\sigma\in W\cap \Alt(\Lambda)\le M$.
 
 Let $h:=m_1\sigma=(1,t,1,\ldots,1)(0\,1)$. By construction, $h\in W\cap \Alt(\Lambda)\le M$. Moreover, $h^2=(t,t,1,\ldots,1)\neq 1$ and $h^4=1$, that is, $h$ has order $4$.
\end{proof}

Next we deal with the case where $M$ is affine.
\begin{lem}\label{l: M aff}
If $M$ is of type ${\rm\bf (Aff)}$, then either the action of $G$ on the cosets of $M$ admits a beautiful subset, or  $G=\Alt(n)$ and $n$ is an odd prime number with $n\ge 13$.
\end{lem}
\begin{proof}
In this case $\Lambda=\{1,\ldots,n\}$ can be identified with the elements of a vector space $V$ of dimension $\ell$ over a field of prime size $r$ (so $n=r^\ell$), and $M=\AGL_\ell(r)$ when $G=\Sym(n)=\Sym(V)$ and $M=\AGL_\ell(r)\cap \Alt(n)$ when $G=\Alt(n)=\Alt(V)$.

\smallskip

\noindent\textsc{Case 1. }$\ell \ge 2$.

\smallskip

\noindent Observe that $n\ge 5$ and hence $\ell \ge 3$ when $r=2$. Since $\SL_2(r)$ contains elements of order $r-1$ and $r+1$, we conclude that $\SL_2(r)$ contains an element of order $6$ unless $r=2$. Since $\SL_3(2)$ contains an element of order $6$, this implies that $\SL_\ell(r)$ contains an element $h$ of order $6$ in all cases. By~\cite[Theorem~$4.2$]{GPS}, the group element $h$ has a cycle of length $6$ in its action on $V=\{1,\ldots,n\}$. Observe also that $h$ fixes the zero vector of $V$. In particular,  $$h=(\lambda_0)(\lambda_1\,\lambda_2\,\lambda_3\,\lambda_4\,\lambda_5\,\lambda_6)\cdots,$$
for some $\lambda_0,\ldots,\lambda_6\in V=\{1,\ldots,n\}$. 
Consider the element $g\in\Sym(\Lambda)$ having cycle structure
$$g:=
(\lambda_0\,\lambda_1\,\lambda_3\,\lambda_2\,\lambda_5\,\lambda_6\,\lambda_4)$$
A computation gives $g^h=g^3$ and so $\langle h\rangle$ normalizes $\langle g\rangle$, and $\langle g,h\rangle$ is a Frobenius group of order $42$. Set $K:=\langle g,h\rangle$. 

We wish to apply Lemma~\ref{l: beaut} to the elements $g$ and $h$ with $t:=7$ and $k:=7$. By construction $h$ is an element of $M$ of order $6=t-1$, $g$ has order $7=t$ and $\mathrm{fix}_\Lambda(g)=n-7=n-k$. In order to apply Lemma~\ref{l: beaut} we must establish that $g\not\in M$. By Table~\ref{t: fixed points}, $n-7=\mathrm{fix}_\Lambda(g)\leq n/r$ and we conclude that $(r,\ell)\in \{(3,2),(2,3)\}$. Setting these cases to one side, for a moment, Lemma~\ref{l: beaut} allows us to conclude that either $\omega_0^K$ is a beautiful subset of size $7$ (and we are done), or else $M$ contains a non-trivial element $f$ such that $\fix_\Lambda(f)\geq n-14$. In the latter case, Table~\ref{t: fixed points} implies that $n-14\le n/r$, that is, $(r,\ell)\in \{(3,2),(2,3),(2,4)\}$. 

Thus, to complete the proof, we must deal with these three cases. When $(r,\ell)=(2,3)$, the action of $G\ge\Alt(8)\cong \SL_4(2)$ on the right cosets of $\AGL_3(2)=\mathrm{ASL}_3(2)$ is $2$-transitive of degree $15$ and hence we have a beautiful subset of size $15$ (a computation with \texttt{magma} shows that we also have beautiful subsets of size $7$ and $8$); when $(r,\ell)=(3,2)$, a computation with \texttt{magma} reveals that $G$ has beautiful subsets of size $5,7$ and $9$; finally, when $(r,\ell)=(2,4)$, another computer-aided computation yields that $G$ has beautiful subsets of size $7$.

\smallskip

\noindent\textsc{Case 2. }$\ell=1$.

\smallskip
 
 \noindent Here $n=r$. Suppose that $G=\Sym(V)$. Then $M=\AGL_1(r)$ and $r\ge 5$ is odd. Let $h$ be an element of $M$ of order $r-1$ and observe that $h$, viewed as a permutation in $\Sym(V)$, has a cycle of length $r-1$ and a fixed point. Let $x\in G$ with $h^x=h^{-1}$. 
 
 Let $\omega_1$ be the coset $Mx$ and let $\Delta:=\omega_1^M$. Observe that $G_{\omega_0}=M$ and $G_{\omega_0}\cap G_{\omega_1}=M\cap M^x=\langle h\rangle$, hence $|\Delta|=r$. Now, as $M$ is maximal in $G$ and $M\le G_{\Delta}$, we get $G_{\Delta}=M$, and $G^\Delta$ is permutation isomorphic to $M$ in its $2$-transitive action on $V$. Therefore $\Delta$ is a beautiful subset of size $r$.
 
Finally, suppose that $G=\Alt(n)$, $M=\AGL_1(n)\cap\Alt(n)$ and $\Omega$ is the set of right cosets of $M$ in $G$. When $n=5$, $G$ acts $2$-transitively on $\Omega$ and $G\ngeq \Alt(\Omega)$ and hence $\Omega$ is a beautiful subset. When $n\in \{7,11\}$, the group $M$ is not maximal in $G$. Therefore $n\ge 13$.
\end{proof}

Finally we deal with the case where $M$ is almost simple.

\begin{lem}\label{l: M AS}
  If $M$ is of type ${\rm \bf(AS)}$, then either the action of $G$ on the cosets of $M$ admits a subset of size $5$ or $7$, or the socle of $M$ is $\PSL_2(2^{f})$ with $f\ge 2$, or $\PSL_2(q)$ with $q\equiv 3,5\pmod 8$.
\end{lem}
\begin{proof}
Here $G$ is either $\Sym(n)$ or $\Alt(n)$ and $M$ is a maximal subgroup of $G$ with $M$ almost simple. Let $T$ be the socle of $M$ and set $\Lambda:=\{1,\ldots,n\}$.  As usual, let $\Omega$ be the set of right cosets of $M$ in $G$ and denote by $\omega_0$ the coset $M$.

\smallskip

\noindent\textsc{Case 1. }Let $\lambda_0\in \Lambda$. Suppose that the stabilizer $M_{\lambda_0}$ contains an element  $h$ of order $4$.

\smallskip

\noindent Consider the action of $h$ on $\Lambda$. As $h$ has order $4$, $h$ has at least one  cycle of length $4$ in its action on $\Lambda$, say $(\lambda_1\,\lambda_2\,\lambda_3\,\lambda_4)$. Consider $g:=(\lambda_0\,\lambda_1\,\lambda_2\,\lambda_4\,\lambda_3)$. A computation gives $g^h=g^2$ and so $\langle h\rangle$ normalizes $\langle g\rangle$ and $\langle g,h\rangle$ is a Frobenius group of order $20$. Set $K:=\langle g,h\rangle$.

We wish to apply Lemma~\ref{l: beaut} to the elements $g$ and $h$ with $t:=5$ and $k:=5$. By construction $h$ is an element of $M$ of order $4=t-1$, $g$ has order $5=t$ and $\mathrm{fix}_\Lambda(g)=n-5=n-k$. We must show that $g\not\in M$: we use the fact that $\fix_\Lambda(g)=n-5\leq n-2(\sqrt{n}-1)$ by Table~\ref{t: fixed points}. Thus $n\leq 12$. Suppose, for the moment, that $n>12$. Then Lemma~\ref{l: beaut} implies that we have a beautiful subset of size $5$, or else there is a non-trivial element $f\in M$ with 
$\fix_\Lambda(f)\geq n-10$ and so, by Table~\ref{t: fixed points}, $n=10 \leq n-2(\sqrt{n}-1)$ and $n\leq 36$.

To complete the proof we must deal with the case $n\leq 36$. By a computation with \texttt{magma}, we see that, except when $G=\Sym(6)$ and $M=\PGL_2(5)$ (or $G=\Alt(6)$ and $M=\mathrm{PSL}_2(5)$), the group $G$ contains a beautiful subset of size $5$. Finally, when $G=\Sym(6)$ and $M=\PGL_2(5)$  (or $G=\Alt(6)$ and $M=\mathrm{PSL}_2(5)$), the action of $G$ is permutation isomorphic to the action of $G$ on $\{1,2,3,4,5,6\}$, which has no beautiful subset. Observe that the socle of $M$ is $\PSL_2(q)$ with $q=5$, in line with the statement of this lemma.  

\smallskip

\noindent\textsc{Case 2. }Suppose that $M$ contains an element $h$ of order $4$.

\smallskip

\noindent  From Case~1, we may assume that, for $\lambda_0\in \Lambda$, $M_{\lambda_0}$ does not contain elements of order $4$. This allows us to exclude one particular family of primitive actions. In fact, if $M$ is either $\Alt(m)$ or $\Sym(m)$ (with $m\ge 5$) in its natural primitive action on  the $k$-subsets of $\{1,\ldots,m\}$ (with $1<k<m/2$), then $M_{\lambda_0}$ is isomorphic to either $\Sym(k)\times \Sym(m-k)$ or to $(\Sym(k)\times \Sym(m-k))\cap \Alt(m)$. In both cases, provided that $(m,k)\ne (5,2)$, we have $m-k\ge 4$ and hence the group $M_{\lambda_0}$ contains an element of order $4$. Therefore we may exclude these groups $M$ from our analysis here. Observe that when $(m,k)=(5,2)$ we have $n=\binom{m}{k}=10$ and the socle of $M$ is $\Alt(5)\cong \PSL_2(q)$ with $q=5$; according to the statement of this lemma the action of $G$ on $\Omega=G/M$  potentially does not have a beautiful subset.

This excluded case means that the group $M$ is now in case {\rm\bf (AS*)}, and we can apply the stronger result Theorem~1 and Corollary~$1$ from~\cite{GM} that is listed in Table~\ref{t: fixed points}. In other words, we know that if $f$ is a non-trivial element of $M$, then $\fix_\Lambda(f)\leq \frac47 n$.

Suppose first that $n>93$. Consider the action of $h$ on $\Lambda$; let $n_1$ be the number of fixed points of $h$, $n_2$ the number of cycles of length $2$ of $h$ and $n_4$ the number of cycles of length $4$ of $h$. Clearly, $n=n_1+2n_2+4n_4$ and $\mathrm{fix}_\Lambda(h^2)=n_1+2n_2=n-4n_4$. Thus $n-4n_4\leq \frac47n$ and, since $n> 93$, we have $n_4\geq 5$. Say that, in its cycle decomposition, 
$$h=(\lambda_1\,\lambda_2\,\lambda_3\,\lambda_4)
(\lambda_5\,\lambda_6\,\lambda_7\,\lambda_8)
(\lambda_9\,\lambda_{10}\,\lambda_{11}\,\lambda_{12})
(\lambda_{13}\,\lambda_{14}\,\lambda_{15}\,\lambda_{16})
(\lambda_{17}\,\lambda_{18}\,\lambda_{19}\,\lambda_{20})\cdots.$$
Consider the element $g\in\Sym(\Lambda)$ having cycle structure
$$g:=
(\lambda_1\,\lambda_5\,\lambda_9\,\lambda_{13}\,\lambda_{17})
(\lambda_2\,\lambda_{10}\,\lambda_{18}\,\lambda_6\,\lambda_{14})
(\lambda_3\,\lambda_{19}\,\lambda_{15}\,\lambda_{11}\,\lambda_7)
(\lambda_4\,\lambda_{16}\,\lambda_8\,\lambda_{20}\,\lambda_{12}).$$
A computation gives $g^h=g^3$ and so $\langle h\rangle$ normalizes $\langle g\rangle$ and $\langle g,h\rangle$ is a Frobenius group of order $20$. Set $K:=\langle g,h\rangle$.

We wish to apply Lemma~\ref{l: beaut} to the elements $g$ and $h$ with $t:=5$ and $k:=20$. By construction $h$ is an element of $M$ of order $4=t-1$, $g$ has order $5=t$ and $\mathrm{fix}_\Lambda(g)=n-20=n-k$. We must show that $g\not\in M$. Since $g$ fixes $n-20$ elements of $\Lambda$, if $g$ were in $M$, then $n-20\leq \frac47n$ and $n\leq 60$ which contradictions our supposition. Now Lemma~\ref{l: beaut} implies that we have a beautiful subset of size $5$, or else there is a non-trivial element $f\in M$ such that $\fix_\Lambda(f)\geq n-40$. Now $n-40\leq \frac47n$ and so $n\leq 93$ which, again, is a contradiction.

We are left with the case where $n\le 93$. A computation with \texttt{magma} shows that, in this case, except when the socle of $M$ is $\PSL_2(q)$ (with $q=2^f$ or $q\equiv 3,5\pmod 8$) $G$ contains a beautiful subset of size $5$ or $7$. 
 
\smallskip

\noindent\textsc{Case 3. }Suppose that $M$ does not have an element of order $4$.

\smallskip

\noindent By a celebrated theorem of Walter~\cite{walter}, the supposition implies that $T$ is isomorphic to one of the following groups $\PSL_2(2^f)$ with $f\ge 2$, $\PSL_2(q)$ with $q\equiv 3,5\pmod 8$, $J_1$ or $\Ree(3^{2f+1})$ with $f\ge 1$. In the first two cases, the lemma does not say anything: potentially, the action of $G$ on $\Omega=G/M$ does not admit beautiful subsets. Therefore it remains to deal with the case that $T$ is either $J_1$ or $\Ree(3^{2f+1})$.

From~\cite{ATLAS} and~\cite{Ree}, we see that both $J_1$ and $\Ree(q)$ contain an element $h$ of order $6$: the centralizer of an involution in $J_1$ is isomorphic to $C_2\times \Alt(5)$ and the centralizer of an involution in $\Ree(3^{2f+1})$ is isomorphic to $C_2\times \PSL_2(3^{2f+1})$ and both $\Alt(5)$ and $\PSL_2(3^{2f+1})$ contain an element of order $3$.

We claim that $h$ has at least seven cycles of length $6$ on $\{1,\ldots,n\}=\Lambda$. This follows with an easy computation when $T=J_1$: we may construct with \texttt{magma} all primitive permutation representations of $M$ and check that an element of order $6$ always has at least seven cycles of length $6$ (actually, an element of order $6$ always has at least $42$ cycles of length $6$). Suppose then $T=\Ree(3^{2f+1})$. Write $q:=3^{2f+1}$.
Now, from~\cite[Theorem~$1$ and Table~$1$]{LawtherLiebeckSeitz}, we have $$\frac{\mathrm{fix}_\Lambda(x)}{|\Lambda|}<\frac{1}{q^2-q+1},$$
for every $x\in M\setminus\{1\}$.

The number $n_1$ of fixed points of $h$ is $\mathrm{fix}_{\Lambda}(h)$, the number $n_2$ of cycles of length $2$ of $h$ is $(\mathrm{fix}_{\Lambda}(h^2)-\mathrm{fix}_{\Lambda}(h))/2$ and the number $n_3$ of cycles of length $3$ of $h$ is $(\mathrm{fix}_{\Lambda}(h^3)-\mathrm{fix}_\Lambda(h))/3$. Therefore the number of $6$ cycles of $h$ is
\begin{align*}
  \frac{n-(n_1+2n_2+3n_3)}{6}&=\frac{n-(\mathrm{fix}_\Lambda(h^2)+\mathrm{fix}_\Lambda(h^3)-\mathrm{fix}_\Lambda(h))}{6}> \frac{n}{6}-\frac{2n}{6(q^2-q+1)}\\
  &\ge \frac{q^3+1}{6}-\frac{q+1}{3}\ge 7,
\end{align*}
where in the second inequality we used that a faithful permutation representation of $M$ has degree at least $q^3+1$ (see~\cite{Ree}), and the third inequality follows with an easy computation.

In particular, 
\begin{align*}
  h=&(\lambda_1\,\lambda_2\,\lambda_3\,\lambda_4\,\lambda_5\,\lambda_6)
(\lambda_7\,\lambda_8\,\lambda_9\,\lambda_{10}\,\lambda_{11}\,\lambda_{12})
(\lambda_{13}\,\lambda_{14}\,\lambda_{15}\,\lambda_{16}\,\lambda_{17}\,\lambda_{18})
(\lambda_{19}\,\lambda_{20}\,\lambda_{21}\,\lambda_{22}\,\lambda_{23}\,\lambda_{24})\\
&(\lambda_{25}\,\lambda_{26}\,\lambda_{27}\,\lambda_{28}\,\lambda_{29}\,\lambda_{30})
(\lambda_{31}\,\lambda_{32}\,\lambda_{33}\,\lambda_{34}\,\lambda_{35}\,\lambda_{36})
(\lambda_{37}\,\lambda_{38}\,\lambda_{39}\,\lambda_{40}\,\lambda_{41}\,\lambda_{42})
  \cdots.
  \end{align*}
Consider the element $g\in\Sym(\Lambda)$ having cycle structure
\begin{align*}
  g:=&
  (\lambda_1\,\lambda_{31}\,\lambda_{19}\,\lambda_7\,\lambda_{37}\,\lambda_{25}\,\lambda_{13})
(\lambda_2\,\lambda_{32}\,\lambda_{20}\,\lambda_{8}\,\lambda_{38}\,\lambda_{26}\,\lambda_{14})
(\lambda_{3}\,\lambda_{33}\,\lambda_{21}\,\lambda_{9}\,\lambda_{39}\,\lambda_{27}\,\lambda_{15})\\
&(\lambda_{4}\,\lambda_{34}\,\lambda_{22}\,\lambda_{10}\,\lambda_{40}\,\lambda_{28}\,\lambda_{16})
(\lambda_{5}\,\lambda_{35}\,\lambda_{23}\,\lambda_{11}\,\lambda_{41}\,\lambda_{29}\,\lambda_{17})
(\lambda_{6}\,\lambda_{36}\,\lambda_{24}\,\lambda_{12}\,\lambda_{42}\,\lambda_{30}\,\lambda_{18}).
  \end{align*}
A computation gives $g^h=g^3$ and so $\langle h\rangle$ normalizes $\langle g\rangle$ and $\langle g,h\rangle$ is a Frobenius group of order $42$. Set $K:=\langle g,h\rangle$. 
We wish to apply Lemma~\ref{l: beaut} to the elements $g$ and $h$ with $t:=7$ and $k:=42$. By construction $h$ is an element of $M$ of order $6=t-1$, $g$ has order $7=t$ and $\mathrm{fix}_\Lambda(g)=n-42=n-k$. We must prove that $g\not\in M$: if $g$ were an element of $M$, Table~\ref{t: fixed points} would imply that $n-42 \leq 4n/7$ and so $n\leq 98$, which contradicts the fact that the degree a faithful permutation representation of $T$ (being $J_1$ or $\Ree(3^{2f+1})$) is at least $266$.

Now Lemma~\ref{l: beaut} implies that we have a beautiful subset of size $7$, or else $M$ contains a non-trivial element $f$ with $\fix_\Lambda(f) \geq n-84$ and, once again, Table~\ref{t: fixed points} implies that  $n-84 \leq 4n/7$ and so $n\leq 196$. Again, since $196<266$, we have a contradiction.
\end{proof}

\section{Classical socle}\label{s: classical}

In this section we will prove Theorem~B. As in the previous section, our approach divides naturally into two sections: we will produce a classification of those actions within the ambit of Theorem~B that admit a beautiful subset, and a proof of Theorem~B will then be obtained by studying those actions that do {\bf not} admit a beautiful subset. In the current section, this last part is particularly easy as there are only  a finite number of (relatively small) such actions.

It will be convenient to take $S$ to be a  quasisimple classical group -- thus, in what follows, $S$ will be one of $\SL_n(q)$, $\Sp_n(q)$, $\SU_n(q)$, $\Omega_n(q)$ (with $n$ odd) or $\Omega_n^\varepsilon(q)$ (with $n$ even and $\varepsilon\in\{+,-\}$). We consider an action of $S$ on some set $\Omega$. Note that we {\bf do not} assume that $S$ acts primitively on $\Omega$, and we drop the (implicit) assumption that $S$ acts faithfully on the set $\Omega$.

Using isomorphisms between classical groups of small dimension, we assume that $n\geq 3$ for $S=\SU_n(q)$, $n\geq 4$ for $S=\Sp_n(q)$, $n\geq 7$ for $S=\Omega_n(q)$ with $n$ odd, and $n\geq 8$ for $S=\Omega^\varepsilon_n(q)$ with $n$ even and with $\varepsilon\in\{+,-\}$. The requirement that $S$ is quasisimple will, in addition, exclude $\SL_2(2)$, $\SL_2(3)$ and $\SU_3(2)$. %In the tables that follow, if a group is isomorphic to a symplectic or alternating group, then it is highlighted as these are already dealt with.

The group $S$ is a subgroup of the group of isometries of some fixed form $\varphi$. We will write $V$ for the associated vector space of dimension $n$ over the field $\K$ where $\K=\Fq$ except when $S$ is a unitary group, in which case $\K=\mathbb{F}_{q^2}$. The form $\varphi$ is either non-degenerate or the zero form (in the case $S=\SL_n(q)$). 

When $\varphi$ is non-degenerate, we will make use of a {\bf hyperbolic basis} $\B$ of $V$ of the form \[\{e_1,\dots, e_k, f_1,\dots, f_k\}\cup \mathcal{A},\] where $k$ is the Witt index of $\varphi$, $\langle e_i, f_i\rangle$ are hyperbolic lines for $i=1,\dots, k$ and $\mathcal{A}$ is either empty or of size at most $2$ and spans an anisotropic subspace of $V$.

Let $W$ be a subspace of $V$. Recall that a {\bf Singer cycle} $C$ on $W$ is a maximal cyclic irreducible subgroup of $\GL(W)$. In particular $C$ has order $|\K|^m-1$ where $m=\dim_{\mathbb{K}}(W)$. The following lemma concerning Singer cycles will be useful. (We denote by $A^T$ the transposed of a matrix $A$ and, when $A\in \GL_n(q^2)$, by $\bar{A}$ the matrix obtained from $A$ applying entry-wise the involutory field  automorphism of $\mathbb{F}_{q^2}$.)

\begin{lem}\label{l: singers}
Let $W=\langle e_i \mid i\in I\rangle$ for some subset $I\subseteq \{1,\dots, k\}$ and let $S_W$ be the set-wise stabilizer of $W$. Then $S_W$ contains a Singer cycle in its action on $W$ in the following cases:
\begin{enumerate}
 \item $S=\Sp_n(q)$;
 \item $S=\SU_n(q)$ with $n$ odd;
 \item $S=\SU_n(q)$ with $n$ even and $\dim(W)<\frac{n}{2}$;
 \item $S=\Omega_n^\varepsilon(q)$ with $n$ even and $\dim(W)<\frac{n}{2}$;
 \item $S=\Omega_n(q)$ with $n$ odd and $\dim(W)<\frac{n-1}{2}$.
% \item $S=\Omega_n^+(q)$ with $q=2^a$ and $\dim(W)<\frac{n}{2}$;
% \item $S=\Omega_n^-(q)$ with $q=2^a$ and $\dim(W)<\frac{n}{2}-1$.
\end{enumerate}
In addition, if $S=\SU_2(q)$ and $W=\langle e_1\rangle$, then $S_W$ induces a $\Fq$-Singer cycle on the $\Fq$-span of $e_1$.
\end{lem}
\begin{proof}
We prove this lemma case by case.
 
Let $S=\Sp_n(q)$. Writing a hyperbolic basis for $V$ with the ordering $e_1,\dots, e_k, f_1,\dots, f_k$, we see that the matrix 
\[
 \begin{pmatrix}
    A & 0 \\ 0 & (A^{-1})^{T}
   \end{pmatrix}
\]
lies in $S$, for every $A\in\GL_{k}(q)$. Thus, for any subspace $W$ of $\langle e_1,\dots, e_k\rangle$, $S_W$ induces $\GL(W)$ on $W$ and hence $S_W$ contains a Singer cycle on $W$.

Let $S=\SU_n(q)$ with $n$ odd and let $\mathcal{A}=\{x\}$. Then, writing a hyperbolic basis for $V$ with the ordering $e_1,\dots, e_k, x,f_1,\dots, f_k$, we see that $S$ contains the matrix 
\[
 \begin{pmatrix}
    A & & 0 \\ & a^{q-1} & \\ 0 & &(\bar{A}^{-1})^{T}
   \end{pmatrix}
\]
where $a=\det(A)$, for every $A\in \GL_{k}(q^2)$. Then the result follows immediately as for $\Sp_n(q)$. If $S=\SU_n(q)$ with $n$ even, then the result follows (from the case $n$ odd) by observing that any even-dimensional non-degenerate proper subspace of $V$ can be extended to an odd-dimensional non-degenerate subspace. The final remark concerning $\SU_2(q)$ can be verified directly.

Let $S=\Omega_n^\varepsilon(q)$ with $n$ even. Observe first that we can extend $W$ to a non-degenerate $(n-2)$-dimensional subspace $W_1$ of $V$ of type $+$. Then the stabilizer in $\SOr_n^\varepsilon(q)$ of $W_1$ is isomorphic to a subgroup of $\SOr_{n-2}^+(q) \times \SOr_2^\varepsilon(q)$. 

From~\cite{kl}, we observe that $|\SOr_2^\varepsilon(q): \Omega_2^\varepsilon(q)|=2$. If $q$ is odd, then this implies that $\SOr_2^\varepsilon(q)$ contains elements of non-square spinor norm. If $q$ is even, then this implies that $\SOr_2^\varepsilon(q)$ contains elements that are a product of an odd number of reflections (see Descriptions 1 and 2 in \cite[pp. 29 and 30]{kl}). In either case, provided $\Omega \neq \Omega_4^+(2)$, one concludes immediately that the set-wise stabilizer of $W_1$ in $\SOr_{n}^\varepsilon(q)$ induces $\SOr_{n-2}^+(q)$ on $W_1$. (If $\Omega=\Omega_4^+(2)$, then there is nothing to prove so this case, too, is covered.)

Now, a matrix of the type used in the symplectic case implies that the set-wise stabilizer of $W_1$ in $\SOr_{n-2}^+(q)$ induces $\GL(W_1)$ on $W_1$. The result then follows immediately, and (arguing exactly as in the unitary case) is easily extended to account for $S=\Omega_n(q)$ with $n$ odd. 
\end{proof}

\subsection{Our method}\label{s: classical method}

Our method for constructing a beautiful subset is, roughly, as follows. Suppose that we are considering the action of $S$ on the set of cosets of a subgroup $M$. 

Suppose, first, that $S$ contains a maximal split torus $T$ with $T\leq M$. It is clear that $M$ cannot contain all of the root subgroups of $S$ (with respect to $T$) as these generate the group $S$. Let $U$ be a root subgroup that is not contained in $M$, and observe that $T$ normalizes $U$.

Suppose, second, that $T$ acts transitively and fixed-point-freely on the non-identity elements of $U$. Then, since $U$ is not contained in $M$, we conclude that $U\cap M=1$. Thus the semi-direct product $H=U\rtimes T$ is endowed of a  $2$-transitive action on $\{Mh\mid h\in H\}$, the action having degree equal to $|U|$. Now Lemma~\ref{l: 2t} implies that $H$ acts 2-transitively on the orbit $\Lambda=M^H=\{Mh\mid h\in H\}$.

To complete the construction we must show that $G^\Lambda\ne \symme(\Lambda)$. Our method for doing this depends on the specific geometry associated with the stabilizer $M$; in general, though, there is a natural monomorphism from $G^\Lambda$ to either $\GL(W)$ or $\PGL(W)$, where $W$ is some section of $V$. Very often $\dim_{\mathbb{K}}(W)=2$ and one concludes immediately (using the subgroup structure of $\GL_2(\mathbb{K})$ and $\PGL_2(\mathbb{K})$) that the largest symmetric group in $G^\Lambda$ has degree $5$. If $|\Lambda|\geq 6$, then we are done.

The method just described depends, of course, on the validity of our two suppositions, as well as our ability to demonstrate that the full symmetric group is not induced by the set-stabilizer $G^\Lambda$. For most classical groups this method works, however for some ``small'' cases all of these requirements fail at various stages, and we will adjust our method accordingly, or ultimately for ``very small'' cases we rely on the invaluable help of the computer algebra system \texttt{magma}~\cite{magma}. This method can also be applied to the other ``geometric'' subgroups of the classical groups (families $\mathcal{C}_2$ to $\mathcal{C}_8$), as well as to the exceptional groups. See~\cite{gls_binary} for more details.

\subsection{Classifying  beautiful subsets}

In this case $M$ is either the stabilizer of a subspace $W$ of $V$ (in which case we write $M=\stab_S(W)$), or the stabilizer of a pair of subspaces, $W_1$ and $W_2$ of $V$ (in which case we write $M=\stab_S(W_1,W_2)$. 

When $S\neq \SL_n(q)$ and $M=\stab_S(W)$, we must distinguish between the case where the subspace is totally isotropic (in which case the group $M$ is a maximal parabolic subgroup), and the case where the subspace is non-degenerate.

 \begin{center}
\begin{table}[!htbp]
\begin{tabular}{ccl}
\toprule[1.5pt]
Line&Group & Details of action \\
\midrule[1.5pt]
1&$\SL_2(4)$& $M=\stab_S(W)$, $\dim(W)=1$\\
2&$\SL_3(2)$&$M=\stab_S(W_1,W_2)$, $V=W_1\oplus W_2$, $\dim(W_1)=1$, $\dim(W_2)=2$\\
3&$\SL_3(2)$, $\SL_3(3)$&$M=\stab_S(W_1,W_2)$, $W_1< W_2$, $\dim(W_1)=1$, $\dim(W_2)=2$\\\midrule[1pt]
4&$\SU_4(2)$ & $M=\stab_S(W)$, $W$ totally isotropic, $\dim(W)=2$ \\
5&$\SU_4(2)$& $M=\stab_S(W)$, $W$ non-degenerate, $\dim(W)=1$ \\
6&$\SU_5(2)$&$M=\stab_S(W)$, $W$ non-degenerate, $\dim (W)=2$\\\midrule[1pt]
7&$\Sp_4(2)', \Sp_4(3)$ & $M=\stab_S(W)$, $W$ totally isotropic, $\dim(W)=1$\\
8&$\Sp_4(2)'$ & $M=\stab_S(W)$, $W$ totally isotropic, $\dim(W)=2$\\\midrule[1pt]
9&$\Omega_7(3)$ & $W$ non-degenerate, $\dim(W)=1$, $M$ of type $\mathrm{O}_6^-(3)\times \mathrm{O}_1(3)$\\
\bottomrule[1.5pt]
\end{tabular}
\caption{Cases where a beautiful subset does not exist.}\label{t: c1 sln}
\end{table}
\end{center}

\begin{prop}\label{prop1}Let $S=\SL_n(q)$ with $n\ge 2$ and $q\ge 4$. Then $S$, in its action on the set $\Omega$ of right cosets of a subgroup $M$ in the Aschbacher family $\mathcal{C}_1$, has a beautiful subset except for the cases  in Line~$1$,~$2$ and~$3$ of Table~$\ref{t: c1 sln}$.
  \end{prop}

\begin{proof}%[Proof of Table~$\ref{t: c1 sln}$]
  According to~\cite{kl}, there are two cases to consider here: either $M$ is the stabilizer of a subspace $W$ of $V$ with $0<W<V$, or $M$ is the stabilizer of a pair $\{W_1,W_2\}$ of subspaces of $V$ with $W_1< W_2$ and $n=\dim(W_1)+\dim(W_2)$ or with $V=W_1\oplus W_2$.

  \smallskip

  \noindent\textsc{Case 1: }Here $M$ is the stabilizer of the subspace $W$ of $V$.

  \smallskip

  \noindent Now $M$ is a maximal parabolic subgroup of $S$. 
Since the action of $S$ on the $k$-dimensional subspaces of $V$ is permutation isomorphic to the action on the $(n-k)$-subspaces of $V$, we may assume that $\dim (W)\le n/2$. 

If $\dim(W)=1$, then the action of $S$ on $\Omega$ is $2$-transitive. Now $\Omega$ itself is a beautiful subset unless $S=\SL_2(4)\cong \alter(5)$: this case is in Line~1 of Table~\ref{t: c1 sln}. 
  
 Suppose next that $\dim(W)>1$. Observe that this implies that $n\geq 4$. Let $W'$ be a subspace of  $W$ with $\dim(W')=\dim(W)-1$ and consider $\Lambda=\{W''\le V\mid W'\subset W'',\dim(W'')=\dim(W)\}$. Clearly, $S_\Lambda=\stab_S(W')$ and the action of $S^\Lambda$ on $\Lambda$ is permutation isomorphic to the natural $2$-transitive action of $\GL(V/W')$ on the $1$-dimensional subspaces of $V/W'$. Therefore $\Lambda$ is a beautiful subset: the exception $\SL_2(4)$ in the previous paragraph does not arise here because $\dim(V/W')>2$.

 \smallskip

 \noindent\textsc{Case 2: }$M$ is the stabilizer of  the pair $\{W_1,W_2\}$, where $W_1$ and $W_2$ are two subspaces of $V$ with $n=\dim(W_1)+\dim(W_2)$ and $W_1< W_2$.

 \smallskip

 \noindent  Here, $\dim(W_1)<n/2$ by \cite[Table 3.5.A]{kl}; in particular, $n\ge 3$. Let $H=\stab_S(W_2) $ be the stabilizer in $S$ of $W_2$. Observe that the action of $H$ on $\{W_1,W_2\}^H$ is permutation isomorphic to the action of  $\GL(W_2)$ on $W_1^H$. In particular, by the previous discussion, when $(\dim(W_2),q)\notin\{(2,2),(2,3),(2,4)\}$, there exists $\Lambda'\subseteq W_1^H$ such that $\Lambda'$ is a beautiful subset for the action of $H$ on the subspaces of $W_2$ of dimension $\dim(W_1)$. Now, set $\Lambda=\{\{W_1',W_2\}\mid W_1'\in \Lambda'\}$. Clearly, $S_\Lambda\subseteq \stab_S(W_2)=H$ and hence $S_{\Lambda}=H_{\Lambda'}$; it follows that $\Lambda$ is a beautiful subset of $S$. It remains to consider the case that $\dim(W_2)=2$ and $q\in \{2,3,4\}$. As $\dim(W_1)<n/2$ and $\dim(W_2)= 2$, we get $n=3$. When $n=3$ and $q=4$, a direct computation with the invaluable help of \texttt{magma} shows that $\SL_3(4)$ admits a beautiful subset of size $9$. The groups $\SL_3(2)$ and $\SL_3(3)$ are listed in Line~3 of Table~$\ref{t: c1 sln}$.

 \smallskip

 \noindent\textsc{Case 3: }$M$ is the stabilizer of  the pair $\{W_1,W_2\}$, where $W_1$ and $W_2$ are two subspaces of $V$ with $V=W_1\oplus W_2$.
 
\smallskip  

\noindent  Here, $\min\{\dim(W_1),\dim(W_2)\}<n/2$ by \cite[Table 3.5.A]{kl}. Replacing $(W_1,W_2)$ by $(W_2,W_1)$ if necessary, we may assume that $\dim(W_1)< \dim(W_2)$.
 Let $\B=(v_1,\dots, v_m)$ be a basis of $W_1$ and let $(v_{m+1},\ldots,v_{n})$ be a basis of $W_2$.
 
Suppose that $\dim(W_1)=1$. We 
take $U$ to be the subgroup of $S$ whose elements fix $v_2, v_3,\dots, v_n$ and satisfy
 \[ 
 v_1\mapsto v_1+ \alpha_2 v_{2}+\alpha_3 v_{3}+\cdots +\alpha_n v_n,
 \]
 for some $\alpha_2,\ldots,\alpha_n\in\Fq$. We take $T$ to be the stabilizer in $S$ of the direct sum decomposition 
 \[
  \langle v_1\rangle \oplus \langle v_{2}, v_3,\ldots,v_n\rangle.
 \]
Then $U$ is an elementary abelian  group of size $q^{n-1}$, $T\cong \GL_{n-1}(q)$ and $\langle U,T\rangle=U\rtimes T$ is isomorphic to the affine general linear group $\AGL_{n-1}(q)$; moreover   the action of $U\rtimes T$ on $$\Lambda=\{W_1,W_2\}^U=\{\{W_1',W_2\}\mid W_1' \textrm{ subspace of }V \textrm{ with }V=W_1'\oplus W_2\}$$ is permutation isomorphic to the natural $2$-transitive affine action of $\AGL_{n-1}(q)$. Observe that $U\rtimes T$ is the stabilizer in $S$ of the subspace $W_2$, therefore $U\rtimes T$ is maximal in $S$ and hence $U\rtimes T=S_\Lambda$; it follows that, for $(n,q)\neq (3,2)$,  $\Lambda$ is a beautiful subset. (When $q=2$ and $n=3$, we have $|\Lambda|=4$, $\AGL_{3-1}(2)\cong \alter(4)$ and $\Lambda$ is not beautiful.) Moreover, the group $\SL_3(2)$ is listed in Line~2 of Table~$\ref{t: c1 sln}$.

Finally, suppose that $\dim(W_1)>1$.  Observe that $\dim(W_2)\ge 3$ because $\dim(W_1)<n/2$ and $\dim(W_1)<\dim(W_2)$. (Recall that $m=\dim(W_1)$.) In this case we proceed similarly to the previous case, but take $U$ to be the subgroup whose elements fix $v_1,\dots, v_{m-1}, v_{m+1},\dots, v_n$ and satisfy
 \[ 
 v_m\mapsto v_m+ \alpha v_{m+1}+\beta v_{m+2}+\gamma v_{m+3},
 \]
 for some $\alpha,\beta,\gamma \in\Fq$. Let $T$ be a subgroup of the stabilizer in $S$ of the direct sum decomposition
 \[
  \langle v_1\rangle \oplus \cdots \oplus \langle v_{m}\rangle \oplus \langle v_{m+1}, v_{m+2}, v_{m+3}\rangle \oplus \langle v_{m+4}\rangle \oplus \cdots \oplus \langle v_n\rangle,
 \]
with $T$ fixing $v_2,\ldots,v_{m},v_{m+4},\ldots,v_{n}$ and acting on $\langle v_{m+1}, v_{m+2},v_{m+3}\rangle$ as a non-split torus of order $q^3-1$: observe that this is possible because, if the element $x\in T$ induces on $\langle v_{m+1},v_{m+2},v_{m+3}\rangle$ a matrix $x'$, the element $x$ induces on $\langle v_1\rangle$ the scalar  $\det (x')^{-1}$, thus guaranteeing that $x\in S=\SL_n(q)$. Then $U$ is a group of size $q^3$ and $U\rtimes T$ acts $2$-transitively on $\Lambda=W^U$, a set of size $q^3$. What is more $S^\Lambda$ is a subgroup of $\PGL(\langle v_{m}, v_{m+1}, v_{m+2}, v_{m+3}\rangle)=\PGL_4(q)$; hence $\Lambda$ is a beautiful subset.
 \end{proof}

\begin{prop}\label{prop2}Let $S=\SU_n(q)$ with $n\ge 3$ and $(n,q)\ne (3,2)$. Then $S$, in its action on the set $\Omega$ of right cosets of a subgroup $M$ in the Aschbacher family $\mathcal{C}_1$, has a beautiful subset except for the cases in Line~$4$,~$5$ and~$6$ of Table~$\ref{t: c1 sln}$.
  \end{prop}

\begin{proof}%[Proof of Table~$\ref{t: c1 sun}$]
Let $\B=(e_1,\dots, e_k, f_1,\dots, f_k, x)$ be a hyperbolic basis for $V$ (we omit the element $x$ if $n$ is even). We also let $E=\langle e_1,\ldots,e_k\rangle$ and $F=\langle f_1,\ldots,f_k\rangle$, recall that the span is over $\mathbb{K}=\mathbb{F}_{q^2}$.

As for the ``$\SL_n(q)$ analysis'', we have two cases to consider here, but of a different nature.

\smallskip

\noindent\textsc{Case 1: }$M$ is the stabilizer of a totally isotropic subspace $W$ of $V$.

\smallskip

\noindent Without loss of generality, we may assume that $W\leq E$, see~\cite{kl}.

\smallskip

\noindent\textsc{Case 1a: }$W<E$ and $(\dim(E),q)\neq (2,2)$.

\smallskip

\noindent From Table~$\ref{t: c1 sln}$, we see that $\SL(E)=\SL_k(q^2)$, in its action on the $\mathbb{F}_{q^2}$-subspaces of $E$ of dimension $\dim_{\mathbb{F}_{q^2}}(W)$, has a beautiful subset $\Lambda$. In particular, $\Lambda$ is a family of totally isotropic subspaces of $E$ of the same dimension of $W$. Consider $E'=\langle W'\mid W'\in \Lambda\rangle$. Clearly, $S_\Lambda\leq \stab_S(E')$ and observe that, since $E'\leq E$, the group $\stab_S(E')$ induces $\GL(E')$  on $E'$. From this it follows immediately that the action induced by $S_\Lambda$ on $\Lambda$ is permutation isomorphic to the action of $\GL(E)$ on $\Lambda$; therefore $\Lambda$ is a beautiful subset of $S$.
 
\smallskip

\noindent\textsc{Case 1b: } $W<E$ and $(\dim (E),q)=(2,2)$.

\smallskip

\noindent Then $\dim(W)=1$ and $S=\SU_4(2)$ or $S=\SU_5(2)$. Let $H=\stab_S(\langle e_1+\alpha f_1\rangle)$, where $\alpha\in \mathbb{F}_4\setminus \mathbb{F}_2$, and observe that $H$ is the stabilizer of a non-isotropic $1$-dimensional subspace of $V$. Now, an easy computation with \texttt{magma} shows that $H$ has an orbit $\Lambda$ of size $9$, $|S^\Lambda|=216$ and the action of $H$ on $\Lambda$ is $2$-transitive; therefore $\Lambda$ is a beautiful subset of $S$. A similar computer computation yields that $\SU_5(2)$ has a beautiful subset of size $9$.

\smallskip

\noindent\textsc{Case 1c: }$W=E$ and $n=2k+1$.

\smallskip

\noindent Let $W'$ be a subspace of $W$ with $\dim(W')=k-1$ and set $$\Lambda=\{W''\leq V\mid \dim(W'')=k,W'\leq W'',\,W''\mathrm{\,totally\,isotropic} \}.$$
Observe that, since each element $W''\in \Lambda$ is totally isotropic, we have $W'\leq W''\le (W'')^\perp\le (W')^\perp$. Now $\dim((W')^\perp)=n-\dim (W')=n-k+1$, $\dim((W')^\perp/W')=n-2k+2=3$ and the Hermitian form on $V$ induces a non-degenerate Hermitian form on $(W')^\perp/W'$. Moreover, for each $W''\in \Lambda$, $W''/W'$ is a $1$-dimensional isotropic subspace of $(W')^\perp/W'$ and hence 
\[|\Lambda|=
q^3+1>5. 
\]
Consider $K=\stab_S((W')^\perp/W')$. For $n=3$, we have $K=S=\SU_3(q)$ because $W'=0$ and $(W')^\perp=V$. For $n>3$, an immediate application of Witt's lemma shows that $K$ induces the group $\GU((W')^\perp/W')$ in its natural action on $(W')^\perp/W'$. From~\cite{kl}, we see that both $\SU_3(q)$ and $\GU_3(q)$ act $2$-transitively on the set of totally isotropic $1$-dimensional subspaces; therefore $\Lambda$ is a beautiful subset of $S$.% provided that $(\dim((W')^\perp/W'),q)\notin \{(2,2),(2,3),(2,4)\}$: recall that $\PSU_2(q)\cong\PSL_2(q)$. Observe that when $\dim((W')^\perp/W')=2$ we have $n=2k$. 

\smallskip

\noindent\textsc{Case 1d: }$W=E$ and $n=2k$.

\smallskip

\noindent The group $S=\SU_{2k}(q)$ contains the extension field subgroup $S'=\SU_2(q^k)$ and its normalizer $H=\Norm S{S'}$: replacing $S'$ by a suitable $S$-conjugate we may view $E$ as  a $1$-dimensional totally isotropic subspace over the extension field $\mathbb{F}_{q^{2k}}$. Let $\Lambda=E^{H}$ and observe that $S_\Lambda=H$, moreover, the action of $H$ on $\Lambda$ is permutation isomorphic to the $2$-transitive action of the $2$-dimensional unitary group $\SU_2(q^k)\cong \SL_2(q^k)$ on the totally isotropic subspaces. Thus $\Lambda$ is a beautiful subset of $S$ provided that $|\Lambda|=q^k+1>5$, that is, $(n,q)\neq (4,2)$. The group $\SU_4(2)$ is listed in Line~4 of Table~\ref{t: c1 sln}.

\smallskip

\noindent \textsc{Case 2: }$M=\stab_S(W)$, where $W$ is a non-degenerate subspace of $V.$

\smallskip

\noindent Observe that \cite[Table 3.5.B]{kl} implies that $\dim(W)< \frac{n}{2}$.

\smallskip

\noindent \textsc{Case 2a: }$\dim (W)=1$.

\smallskip

\noindent Write $y=x$ when $n=3$, and $y=e_2+\alpha f_2$ (for a chosen $\alpha\in \mathbb{F}_{q^2}$ with $\alpha+\alpha^q\neq 0$) when $n>3$. Now consider $W'=\langle e_1,y,f_1\rangle$ and observe that the Hermitian form on $V$ restricts to a non-degenerate Hermitian form on the $3$-dimensional vector space $W'$. Now, consider the subgroup $H$ of $\stab_S(W')$ that induces on $W'$ (with basis given by $e_1,y,f_1$) the matrices of the form
\[\begin{pmatrix}c&-cb^{q}&ca\\0&c^{q-1}&c^{q-1}b\\0&0&c^{-q}\end{pmatrix} \,\textrm{with }a,b,c\in\mathbb{F}_{q^2},\, a+a^q+b^{q+1}=0\textrm{  and  }c\ne 0.
\]
Observe that the matrix group induced by $H$ on $W'$ has order $q^3(q^2-1)$ and  is the stabilizer of the $1$-dimensional totally isotropic subspace $\langle e_1\rangle$. (In other words, we are defining $H=\stab_S(W')\cap \stab_S(\langle e_1\rangle)$.) Moreover, $\Lambda=\{\langle \gamma e_1+y\rangle\mid \gamma\in\mathbb{F}_{q^2}\}$ is a set of $q^2$ non-degenerate $1$-dimensional subspaces of $V$ contained in $\langle e_1,y\rangle\leq W'$. By construction $H$ acts transitively on $\Lambda$, moreover the stabilizer $K$ in $H$ of the element $\langle y\rangle$ induces on $W'$ the matrix group formed by the diagonal matrices:
\[
\begin{pmatrix}c&0&0\\0&c^{q-1}&0\\0&0&c^{-q}\end{pmatrix}\,\textrm{with }c\in\mathbb{F}_{q^2}\setminus\{0\}.
\]
A quick computation shows that $\langle e_1+y\rangle^K=\left\{\langle c^{-q+2}e_1+y\rangle\mid c\in \mathbb{F}_{q^2}\setminus\{0\}\right\}$. Therefore $H$ acts $2$-transitively on $\Lambda$ provided that this set equals $\Lambda\setminus\{y\}$, that is, $\{c^{-q+2}\mid c\in \mathbb{F}_{q^2}\setminus\{0\}\}=\mathbb{F}_{q^2}\setminus\{0\}$. Clearly, this happens only when $\gcd(q^2-1,q-2)=1$. It is easy to see that $\gcd(q^2-1,q-2)=1$ when $3\nmid q+1$. In particular, when $3\nmid q+1$, $\Lambda$ is a beautiful subset for the action of $H$ on the non-degenerate $1$-dimensional subspaces of $W'$. From this, it follows easily that $\Lambda$ is also a beautiful subset of $S$.

Assume now that $\dim (W)=1$, $3\mid q+1$ and $q\neq 2$. Now the argument is exactly the same as above, but considering the subgroup $H$ of $\stab_S(W')$ that induces on $W'$ (with basis given by $e_1,y,f_1$) the matrices of the form
\[\begin{pmatrix}c&-cb&ca\\0&1&b\\0&0&c^{-1}\end{pmatrix} \,\textrm{with }b,c\in\mathbb{F}_{q},\,a\in \mathbb{F}_{q^2},\,\, a+a^q+b^2=0\textrm{  and  }c\ne 0.
\]
Observe that the matrix group induced by $H$ on $W'$ has order $q^2(q-1)$ and that $\Lambda=\{\langle \gamma e_1+y\rangle\mid \gamma\in\mathbb{F}_{q}\}$ is a set of $q$ non-degenerate $1$-dimensional subspaces of $V$ contained in $\langle e_1,y\rangle\leq W'$. By construction $H$ acts transitively on $\Lambda$, moreover the stabilizer $K$ in $H$ of the element $\langle y\rangle$ induces on $W'$ the matrix group formed by the diagonal matrices:
\[
\begin{pmatrix}c&0&0\\0&1&0\\0&0&c^{-q}\end{pmatrix}\,\textrm{with }c\in\mathbb{F}_{q}.
\]
Another computation shows that $\langle e_1+y\rangle^K=\left\{ce_1+y\mid c\in \mathbb{F}_{q}\setminus\{0\}\right\}$. Therefore $H$ acts $2$-transitively on $\Lambda$. In particular, when $q\ne 2$, we have $|\Lambda|=q+1\ge 5$ and $\Lambda$ is a beautiful subset for the action of $H$ on the non-degenerate $1$-dimensional subspaces of $W'$. From this, it follows easily that $\Lambda$ is also a beautiful subset of $S$.

Assume now that $\dim (W)=1$ and $q=2$. As $\SU_3(2)$ is soluble, we may assume that $n>3$: we observe here that $\SU_3(2)$ has no beautiful subset. Moreover, with a calculation with \texttt{magma} we see that also $\SU_4(2)$ has no beautiful subset: the group $\SU_4(2)$ is listed in Line~5 of Table~\ref{t: c1 sln}. Consider, for a moment, the case $n=5$, and let $\Lambda=\{\langle x+v\rangle\mid v\in \langle e_1,e_2\rangle\}$. Then $|\Lambda|=|\langle e_1,e_2\rangle|=4^2=16$. Now, a computation in $\SU_5(2)$ gives that $S_\Lambda=\stab_S(\langle e_1,e_2\rangle)\cap\stab_S(\langle e_1,e_2,x\rangle)=\stab_S(\langle e_1,e_2\rangle)$ and that the action of $S_\Lambda$ on $\Lambda$ is permutation isomorphic to the natural $2$-transitive action of the affine general linear group $\AGL_2(4)$; therefore $\Lambda$ is a beautiful subset of $S$. The construction when $n>5$ is entirely similar and the vector $x$ is replaced by a non-isotropic vector orthogonal to $\langle e_1,e_2\rangle$.

\smallskip

\noindent \textsc{Case 2b: }$\dim (W)>1$.

\smallskip

\noindent For the time being, we exclude the case $(n,q)=(5,2)$. Fix $W'\leq W$ with $\dim (W')=\dim (W)-1$ and $W'$ non-degenerate, and consider $W''=(W')^\perp$. Clearly, $\dim (W'')=n-\dim (W)+1\ge 4$, because $\dim (W)<n/2$ and $\dim (W)>1$. Now, $W''$ is also non-degenerate and $\stab_S(W')=\stab_S(W'')$ induces the matrix group $\GU(W'')$ on $W''$. Observe that from the previous discussion, the group $\SU(W'')$ admits a beautiful subset $\Lambda'$ for the action on the $1$-dimensional non-degenerate subspaces of $W''$, provided that $(\dim (W''),q)\neq (4,2)$. When $(\dim (W''),q)=(4,2)$, we have $(n,q)=(5,2)$ and we are excluding this case for the moment.
Let $\Lambda=\{W'\oplus \langle v\rangle\mid \langle v\rangle\in \Lambda'\}$. By construction, $\Lambda$ consists of non-degenerate subspaces of $V$ of the same dimension of $W$. Moreover, $\bigcap_{U\in\Lambda}U=W'$ and hence $S_\Lambda\leq \stab_S(W')$; from this it follows that $S_\Lambda=\stab_S(W')\cap S_{\Lambda'}$. Furthermore, the action of $S_\Lambda$ on $\Lambda$ is permutation isomorphic to the action of $S_\Lambda$ on $\Lambda'$ and therefore $\Lambda$ is a beautiful subset of $S$.

Finally, when $(n,q)=(5,2)$ and $\dim (W)=2$, a computation with \texttt{magma} shows that $S$ has no beautiful subset: this exception is listed in Line~6 of Table~\ref{t: c1 sln}. 
 \end{proof}

\begin{prop}\label{prop3}Let $S=\Sp_n(q)$ with $n\ge 4$. Then $S$, in its action on the set $\Omega$ of right cosets of a subgroup $M$ in the Aschbacher family $\mathcal{C}_1$, has a beautiful subset except for the cases in Line~$7$ and~$8$ of Table~$\ref{t: c1 sln}$.
  \end{prop}

\begin{proof}%[Proof of Table~$\ref{t: c1 spn}$]
 Let $\mathcal{B}=(e_1,\dots, e_k, f_1,\dots, f_k)$ be a hyperbolic basis for $V$. We also let $E=\langle e_1,\ldots,e_k\rangle$ and $F=\langle f_1,\ldots,f_k\rangle$. As $n\geq 4$, we have $k\geq 2$.

 As for the ``$\SU_n(q)$ analysis'', we have two cases to consider here.

 \smallskip

 \noindent\textsc{Case 1: }$M$ is the stabilizer of a totally isotropic subspace $W$ of $V$.

 \smallskip

 \noindent Without loss of generality, we may assume that $W\leq E$, see~\cite{kl}. The argument in this case is similar to the case ``$S=\SU_n(q)$" and hence we skip the parts of the proof that follow closely the argument for $\SU_n(q)$. 

 \smallskip

 \noindent \textsc{Case 1a: }$W<E$.

 \smallskip

 \noindent For the time being, we exclude the case that $(\dim(E),q)\notin \{(2,2),(2,3),(2,4)\}$. From Table~$\ref{t: c1 sln}$, we see that $\SL(E)=\SL_k(q)$, in its action on the $\mathbb{F}_{q}$-subspaces of $E$ of dimension $\dim_{\mathbb{F}_{q}}(W)$, has a beautiful subset $\Lambda$. In particular, $\Lambda$ is a family of totally isotropic subspaces of $E$ of the same dimension of $W$. Consider $E'=\langle W'\mid W'\in \Lambda\rangle$. Clearly, $S_\Lambda\leq \stab_S(E')$ and observe that, since $E'\leq E$, the group $\stab_S(E')$ induces $\GL(E')$  on $E'$. From this it follows immediately that the action induced by $S_\Lambda$ on $\Lambda$ is permutation isomorphic to the action of $\GL(E)$ on $\Lambda$; therefore $\Lambda$ is a beautiful subset of $S$.
 
Suppose that $W<E$ and $(\dim (E),q)\in \{(2,2),(2,3),(2,4)\}$. Then $\dim(W)=1$ and $S\in \{\Sp_4(2)',\Sp_4(3),\Sp_4(4)\}$. A computation with \texttt{magma} shows that $\Sp_4(2)'$ and $\Sp_4(3)$ have no beautiful subset (these exceptions are in Line~7 of Table~\ref{t: c1 sln}), and $\Sp_4(4)$ has  a beautiful subset of size $16$.

%Suppose that $W=E$. Let $W'$ be a subspace of $W$ with $\dim(W')=k-1$ and set $$\Lambda=\{W''\leq V\mid \dim(W'')=k,W'\leq W,\,W''\mathrm{\,totally\,isotropic} \}.$$
%Observe that, since each element $W''\in \Lambda$ is totally isotropic, we have $W'\leq W''\le (W'')^\perp\le (W')^\perp$. Now $\dim((W')^\perp)=n-\dim (W')=n-k+1$, $\dim((W')^\perp/W')=n-2k+2=2$ and the symplectic form on $V$ induces a non-degenerate symplectic form on $(W')^\perp/W'$. Moreover, for each $W''\in \Lambda$, $W''/W'$ is a $1$-dimensional subspace of $(W')^\perp/W'$ and hence 
%$|\Lambda|=q+1$. 
%Consider $K=\stab_S((W')^\perp/W')=\stab_S(W')=\stab_S((W')^\perp)$. An immediate application of Witt's lemma shows that $K$ induces the group $\SL((W')^\perp/W')$ in its natural action on $(W')^\perp/W'$. Clearly, $\SL_2(q)$ acts $2$-transitively on the set of $1$-dimensional subspaces; therefore $\Lambda$ is a beautiful subset of $S$ provided that $(\dim((W')^\perp/W'),q)\notin \{(2,2),(2,3),(2,4)\}$. Observe that when $\dim((W')^\perp/W')=2$ we have $n=2k$. 

\smallskip

\noindent\textsc{Case 1b: }$W=E$.

\smallskip

\noindent The group $S=\Sp_{2k}(q)$ contains the extension field subgroup $S'=\Sp_2(q^k)$ and its normalizer $H=\Norm S{S'}$: replacing $S'$ by a suitable $S$-conjugate we may view  $E$ as  a $1$-dimensional totally isotropic subspace over the extension field $\mathbb{F}_{q^{k}}$. Let $\Lambda=E^{H}$ and observe that $S_\Lambda=H$; moreover, the action of $H$ on $\Lambda$ is permutation isomorphic to the $2$-transitive action of the $2$-dimensional symplectic group $\Sp_{2}(q^k)$ on the $1$-dimensional subspaces. Thus $\Lambda$ is a beautiful subset of $S$ provided that $|\Lambda|=q^k+1>5$, that is $(n,q)\neq (4,2)$. With \texttt{magma} we see that $\Sp_4(2)'$ has no beautiful subset: this example is in Line~8 of Table~\ref{t: c1 sln}.

\smallskip

\noindent\textsc{Case 2: }$M=\stab_S(W)$, where $W$ is a non-degenerate subspace of $V$.

\smallskip

\noindent Observe that $\dim (W)$ is even and that~\cite[Table 3.5.B]{kl} implies $\dim(W)< \frac{n}{2}$.

\smallskip

\noindent \textsc{Case 2a: }$\dim (W)=2$.

\smallskip

\noindent Observe that $n\geq 6$. Replacing $W$ by a suitable $S$-conjugate we may assume that $W=\langle e_1,f_1\rangle$. Let $H=\stab_S (\langle e_1\rangle)$ and observe that each element $x\in H$ acts as a scalar on $\langle e_1\rangle$, as the inverse of this scalar on $V/\langle e_1\rangle^\perp$ and induces an arbitrary element of $\Sp_{n-2}(q)$ on $\langle e_1\rangle^\perp/\langle e_1\rangle$.
It follows that $\Lambda=W^H=\{\langle e_1,f_1+v\rangle\mid v\in \langle e_2,\ldots, e_k,f_2,\ldots,f_k\rangle\}$ has size $q^{n-2}\ge 2^{6-2}=16$ and that $H$ acts as the affine $2$-transitive group $\mathbb{F}_q^{n-2}\rtimes \Sp_{n-2}(q)$ on $\Lambda$: recall that $\Sp_{2n-2}(q)$ acts transitively on the vectors in $\mathbb{F}_q^{n-2}\setminus\{0\}$, see~\cite{kl}. Moreover, by construction $H$ is maximal in $S$ and hence $S_\Lambda=H$; therefore $\Lambda$  is a beautiful subset.

\smallskip

\noindent\textsc{Case 2b: }$\dim (W)>2$.

\smallskip

\noindent Fix $W'\leq W$ with $\dim (W')=\dim (W)-2$ and $W'$ non-degenerate, and consider $W''=(W')^\perp$. Clearly, $\dim (W'')=n-\dim (W)+2\ge 6$, because $\dim (W)<k$ and $\dim (W)>2$. Now, $W''$ is also non-degenerate and $\stab_S(W')=\stab_S(W'')$ induces the matrix group $\Sp(W'')$ on $W''$. Observe that from the previous discussion, the group $\Sp(W'')$ admits a beautiful subset $\Lambda'$ for the action on the $2$-dimensional non-degenerate subspaces of $W''$. 
Let $\Lambda=\{W'\oplus X\mid X\in \Lambda'\}$. By construction, $\Lambda$ consists of non-degenerate subspaces of $V$ of the same dimension of $W$. Moreover, $\bigcap_{U\in\Lambda}U=W'$ and hence $S_\Lambda\leq \stab_S(W')$; from this it follows that $S_\Lambda=\stab_S(W')\cap S_{\Lambda'}$. Furthermore, the action of $S_\Lambda$ on $\Lambda$ is permutation isomorphic to the action of $S_\Lambda$ on $\Lambda'$ and therefore $\Lambda$ is a beautiful subset of $S$.
 \end{proof}

\begin{prop}\label{prop4}Let $S$ be either $\Omega_n(q)$ with $n\ge 7$, or $\Omega_n^+(q)$ or $\Omega_n^-(q)$ with $n\ge 8$. Then $S$, in its action on the set $\Omega$ of right cosets of a subgroup $M$ in the Aschbacher family $\mathcal{C}_1$, has a beautiful subset except for the case in Line~$9$ of Table~$\ref{t: c1 sln}$.
  \end{prop}

\begin{proof}%[Proof of Tables~$\ref{t: c1 omegan}$, $\ref{t: c1 omegaplusn}$ and $\ref{t: c1 omegaminusn}$]
If $n$ is odd, then let $\B=(e_1,\dots, e_k, f_1,\dots, f_k, x)$ be a hyperbolic basis for $V$, and note that $n\geq 7$, and that $q$ is odd. If $n$ is even and $S=\Omega_n^-(q)$, then let $\B=(e_1,\dots, e_k, f_1,\dots, f_k, x, y)$ be a hyperbolic basis for $V$, and note that $n\geq 8$.
If $n$ is even and $S=\Omega_n^+(q)$, then let $\B=(e_1,\dots, e_k, f_1,\dots, f_k)$ be a hyperbolic basis for $V$, and note that $n\geq 8$. 

In all cases we write $E=\langle e_1,\dots, e_k\rangle$, and note that $E$ is a maximal totally singular subspace of $V$. % Our analysis is very similar to what has gone previously so we will omit some details in what follows.

\smallskip

\noindent\textsc{Case 0: } Some small groups.

\smallskip

\noindent It will be convenient to deal with some small cases first using \texttt{magma} \cite{magma}.

\textbf{ Suppose that $S=\Omega_7(3)$}. Here, the action of $S$ on the $1$-dimensional, $2$-dimensional and $3$-dimensional totally isotropic subspaces of $V$ has beautiful subsets of size $13$.  Moreover, $\Omega_7(3)$ in its action on the $1$-dimensional non-degenerate subspaces with stabilizers of type $\mathrm{O}_6^+(3)\times\mathrm{O}_1(3)$ has beautiful subsets of size $13$. However, $\Omega_7(3)$ in its action on the $1$-dimensional non-degenerate subspaces with stabilizers of type $\mathrm{O}_6^-(3)\times\mathrm{O}_1(3)$ has no beautiful subsets: this exception is listed in Line~9 of Table~\ref{t: c1 sln}. Furthermore, $\Omega_7(3)$ in both of its primitive actions on the $3$-dimensional non-degenerate subspaces (with stabilizers of type $\mathrm{O}_4^-(3)\times\mathrm{O}_3(3)$ or $\mathrm{O}_4^+(3)\times \mathrm{O}_3(3)$) has  beautiful subsets of size $5$. (Observe that from~\cite[Table~3.5C, column VI]{kl}, we may exclude from our study the actions of $\Omega_7(3)$ on the cosets of its subspace stabilizers of type $\mathrm{O}_2^+(3)\times \mathrm{O}_5(3)$ and $\mathrm{O}_2^-(3)\times\mathrm{O}_5(3)$.) 

\textbf{ Suppose that $S=\Omega_8^+(2)$}. In each primitive faithful action of $G=\Omega_8^+(2)$ and $G=\mathrm{SO}_8^+(2)$, we have checked that $G$ has beautiful subsets of size either $5$ or $7$.

\textbf{Suppose that $S=\Omega_8^-(2)$}. Here the group $\Omega_8^-(2)$ has beautiful subsets of size $7$ in its action on isotropic points, isotropic lines and isotropic planes; moreover, it has beautiful subsets of size $7$ on non-isotropic points, and beautiful subsets of size $5$ for the remaining actions on non-degenerate subspaces with stabilizers of type $\mathrm{O}_2^-(2)\times \mathrm{O}_6^+(2)$ or  $\mathrm{O}_4^-(2)\times \mathrm{O}_4^+(2)$. (Observe that from~\cite[Table~3.5F, column VI]{kl}, we may exclude from our study the action of $\Omega_8^-(2)$ on the cosets of its subspace stabilizers of type $\mathrm{O}_2^+(2)\times \mathrm{O}_6^-(2)$.)

\textbf{ Suppose that $S=\Omega_{10}^-(2)$}. 
We verify that $S$ has beautiful subsets of size $5$ in its action on the $2$-dimensional, $3$-dimensional and $4$-dimensional totally isotropic subspaces of $V$, and beautiful subsets of size $7$ in its action on the $1$-dimensional totally isotropic subspaces.

Our analysis from here excludes $\Omega_{7}(3)$, $\Omega_8^+(2)$ and $\Omega_8^-(2)$, and the primitive actions of $\Omega_{10}^-(2)$ on totally isotropic subspaces of $V$ of a given dimension.\footnote{The action of $\Omega_8^-(2)$ on the $2$-dimensional non-degenerate subspaces with stabilizers of type $\mathrm{O}_2^+(2)\times\mathrm{O}_6^-(2)$ is not primitive, see for instance~\cite[Table~3.5F]{kl}.  However, for the proof of \textsc{Case~2c} it is relevant to observe that this imprimitive action does admit a beautiful subset $\Delta$ with the property that $\bigcap_{W'\in\Delta}W'=0$. This fact can be verified with \texttt{magma}.\label{fn123}}

\smallskip

\noindent\textsc{Case 1: }$M$ is the stabilizer of a totally isotropic subspace $W$ of $V$. 

\smallskip

\noindent Write $\dim(W)=m$ and observe that $m\leq k$. Without loss of generality, we may assume that $W\leq E$, see~\cite{kl}. Indeed, we can take $W=\langle e_1, \dots, e_m\rangle$. 

\smallskip

\noindent\textsc{Case 1a: } $q\geq 5$ and $m<k$.

\smallskip

\noindent We define $W'=\langle e_1,\dots, e_{m-1}\rangle$ when $m>1$ and $W'=0$ when $m=1$, and 
\[
 \Lambda = \langle W''\leq E \mid \dim(W'')=\dim(W)=m, W'\subseteq W''\rangle.
\]
Then $S_\Lambda=\stab_S(W'')\cap \stab_S(E)$ and the action of $S^\Lambda$ on $\Lambda$ is permutation isomorphic to the $2$-transitive action of $\GL(E/W'')$ on the $1$-dimensional subspaces of $E/W''$; thus $\Lambda$ is a beautiful subset of $\Omega$.

\smallskip

\noindent\textsc{Case 1b: } $q\geq 5$ and $m=k$.

\smallskip

\noindent Observe first that $m>1$. In this case, define
\begin{align*}
 W_1 &= \langle e_2,\dots, e_m, f_1 \rangle\; \\
 W_2 &= \langle e_1,\dots, e_{m-1}, f_m \rangle.
\end{align*}
Now set
\[
 \Lambda = \left\langle W''\leq V \mid  \begin{array}{l} W'' \textrm{ totally singular}, \dim(W'')=\dim(W)=m, \\
 \dim(W''\cap W_1)=\dim(W''\cap W_2)=m-1\end{array}\right\rangle.
\]
Note that every element of $\Lambda$ contains $W_1\cap W_2=\langle e_2,\dots, e_{m-1}\rangle$.  Now, for each $1$-space $\langle v+(W_1\cap W_2)\rangle$ in the 2-dimensional space $W_1/(W_1\cap W_2)$, there is a $1$-space $\langle w+(W_1\cap W_2)\rangle$ in $W_2/(W_1\cap W_2)$ such that $\langle W_1\cap W_2,v,w\rangle$ is an element of $\Lambda$. It is easy to see that all elements of $\Lambda$ have this form, and so $|\Lambda|=q+1$.

This correspondence between elements of $\Lambda$ and $1$-spaces in $W_1/(W_1\cap W_2)$ also implies that the action of $S_\Lambda$ on $\Lambda$ induces the action of $\GL(W_1/(W_1\cap W_2))$ on $1$-spaces in $W_1/(W_1\cap W_2)$, thus we have a beautiful subset of size $q+1$.

\smallskip

\noindent\textsc{Case 1c: } $q\in\{3,4\}$ and $m<k-1$.

\smallskip

\noindent Thanks to \textsc{Case 0}, we may assume that $n\geq 8$. We take $U$ to be the subgroup of $S$ whose elements fix all elements of $\B$ except $e_m, f_{m+1}$ and $f_{m+2}$ and satisfy
 \[ 
 e_m \mapsto e_m+\alpha e_{m+1}+\beta e_{m+2}
 \]
for some $\alpha, \beta \in \Fq$. Take $\Lambda= W^U$ and take $T$ to be a maximal torus that preserves the subspace decomposition
\[
 \langle e_{m+1}, e_{m+2} \rangle \oplus \langle f_{m+1}, f_{m+2}\rangle \oplus \bigoplus\limits_{\begin{array}{cc} v\in\B \\ v\not\in\{e_{m+1}, e_{m+2}, f_{m+1}, f_{m+2}\}\end{array}} \langle v\rangle
\]
and acts as a Singer cycle on $\langle e_{m+1}, e_{m+2}\rangle$. Note that Lemma~\ref{l: singers} implies that such a torus exists; indeed, applying Lemma~\ref{l: singers} to $\langle e_m, f_m\rangle^\perp$, we see that $T$ contains a subtorus that acts as a Singer cycle on  $\langle e_{m+1}, e_{m+2}\rangle$, and acts trivially on $\langle e_m\rangle$. This implies, in particular, that $U\rtimes T$ acts 2-transitively on $\Lambda$, a set of size $q^2$. Now define
\begin{align*}
 W_1 &= \bigcap\limits_{W'\in \Lambda} W'=\langle e_1,\dots, e_{m-1}\rangle; \\
 W_2 &= \langle W' \mid W'\in\Lambda \rangle = \langle e_1,\dots, e_{m+2}\rangle
\end{align*}
and observe that $S_\Lambda$ stabilizes both $W_1$ and $W_2$. The action of $S_\Lambda$ on $\Lambda$ induces a homomorphism from some subgroup of $\GL(W_2/W_1)=\GL_3(q)$ onto $S^\Lambda$. If $\Lambda$ is not beautiful, then we obtain an epimorphism from a subgroup of $\GL_3(q)$ onto $\Alt(q^2)$, which is impossible (see, for instance,~\cite[Proposition 5.3.7]{kl}).

\smallskip

\noindent\textsc{Case 1d: } $q\in\{3,4\}$ and $m\in \{k-1, k\}$.

\smallskip

\noindent Once again, we may assume that $n\geq 8$. Suppose first that $m\leq 2$. Then $m=2=k-1$ and $S=\Omega_8^-(q)$. Let $W'$ be a subspace of $W$ with $\dim(W')=1$ and consider $$\Delta=\{W''\le V\mid W'< W'', \dim(W'')=2, W''\textrm{ totally isotropic}\}.$$
Observe that, since each element $W'\in \Delta$ is totally isotropic, we have $W'\le W''\le (W'')^\perp\le (W')^\perp$. Now $\dim((W')^\perp)=8-1=7$, $\dim((W')^\perp/W')=8-2=6$ and the orthogonal form on $V$ induces a non-degenerate orthogonal form on $(W')^\perp/W'$ of ``minus'' type. Moreover, for each $W''\in \Delta$, $W''/W'$ is a $1$-dimensional isotropic subspace of $(W')^\perp/W'$. Consider $K=\stab_S((W')^\perp/W')$. An immediate application of Witt's lemma shows that $K$ induces the group $\mathrm{SO}((W')^\perp/W')$ in its natural action on $(W')^\perp/W'$.  Now, recalling that $\Omega_6^-(q)\cong \mathrm{SU}_4(q)$ and that the action of $\Omega_6^-(q)$ on the $1$-dimensional totally isotropic subspaces of $(W')^\perp/W'$ is permutation isomorphic to the action of $\SU_4(q)$ on the $1$-dimensional totally singular subspaces of a $4$-dimensional Hermitian space, we deduce that $\Omega_8^-(q)$ has a beautiful subset from Proposition~\ref{prop2}.

From here, we assume that $m>2$. In this case we take $U$ to be the subgroup of $S$ whose elements fix all elements of $\B$ except $e_m, e_1$ and $e_2$ and satisfy
 \[ 
 e_m \mapsto e_m+\alpha f_1+\beta f_2
 \]
for some $\alpha, \beta \in \Fq$. Take $\Lambda= W^U$ and take $T$ to be a maximal torus that preserves the subspace decomposition
\[
 \langle e_1, e_2 \rangle \oplus \langle f_1, f_2\rangle \oplus \bigoplus\limits_{\begin{array}{cc} v\in\B \\ v\not\in\{e_1, e_2, f_1, f_2\}\end{array}} \langle v\rangle
\]
and acts as a Singer cycle on $\langle e_1, e_2\rangle$. Applying Lemma~\ref{l: singers} as before, we obtain that $U\rtimes T$ acts 2-transitively on $\Lambda$, a set of size $q^2$.  Now define
\begin{align*}
 W_1 &= \bigcap\limits_{W'\in \Lambda} W' = \langle e_3,\dots, e_{m-1}\rangle; \\
 W_2 &= \langle W' \mid W'\in\Lambda\rangle = \langle e_1,\dots, e_m, f_1, f_2,f_m\rangle
\end{align*}
and observe that $S_\Lambda$ stabilizes $W_1$ and $W_2$. The action of $S_\Lambda$ on $\Lambda$ induces a homomorphism from some subgroup of $\GL(W_2/W_1)=\GL_6(q)$ onto $S^\Lambda$. If $\Lambda$ is not beautiful, then we obtain an epimorphism from a subgroup of $\GL_6(q)$ onto $\Alt(q^2)$, which is impossible (see, for instance, \cite[Proposition 5.3.7]{kl}).
\smallskip

\noindent\textsc{Case 1e: } $q=2$ and $m<k-2$.

\smallskip

\noindent Here $n$ is even and, by \textsc{ Case 0}, $n\geq 10$. We take $U$ to be the subgroup of $S$ whose elements fix all elements of $\B$ except $e_m, f_{m+1},f_{m+2}$ and $f_{m+3}$ and satisfy
 \[ 
 e_m \mapsto e_m+\alpha e_{m+1}+\beta e_{m+2}+\gamma e_{m+3}
 \]
for some $\alpha, \beta,\gamma \in \Fq$. Take $\Lambda= W^U$ and take $T$ to be a maximal torus that preserves the subspace decomposition
\[
 \langle e_{m+1}, e_{m+2},e_{m+3} \rangle \oplus \langle f_{m+1}, f_{m+2},f_{m+3} \rangle \oplus \bigoplus\limits_{\begin{array}{cc} v\in\B \\ v\not\in\{e_{m+1}, e_{m+2}, e_{m+3}, f_{m+1}, f_{m+2}, f_{m+3}\}\end{array}}\!\!\!\!\!\!\!\!\!\! \langle v\rangle
\]
and acts as a Singer cycle on $\langle e_{m+1}, e_{m+2}, e_{m+3}\rangle$. Then $U\rtimes T$ acts 2-transitively on $\Lambda$, a set of size $8$. Now define
\begin{align*}
 W_1 &= \bigcap\limits_{W'\in \Lambda} W' = \langle e_1,\dots, e_{m-1}\rangle; \\
 W_2 &= \langle W' \mid W'\in\Lambda\rangle = \langle e_1,\dots, e_{m+3}\rangle
\end{align*}
and observe that $S_\Lambda$ stabilizes $W_1$ and $W_2$. The action of $S_\Lambda$ on $\Lambda$ induces a homomorphism from some proper subgroup $H$ of $\GL(W_2/W_1)=\GL_4(2)$ onto $S^\Lambda$. (The subgroup $H$ is \emph{proper} because, for instance, it stabilizes the 3-space $\langle e_{m+1}, e_{m+2}, e_{m+3},W_1\rangle/W_1$.) If $\Lambda$ is not beautiful, then we obtain an epimorphism from some proper subgroup of $\GL_4(2)$ onto $\Alt(8)$, which is impossible.
\smallskip

\noindent\textsc{Case 1f: } $q=2$ and $m\in \{k-2, k-1, k\}$.

\smallskip

\noindent Once again, we may assume that $n$ is even and $n\geq 10$. Suppose first that $m\leq 3$. Then either $m=3$ and $S\in\{\Omega_{10}^+(2),\Omega_{10}^-(2),\Omega_{12}^-(2)\}$, or $m=2$ and $S=\Omega_{10}^-(2)$. From Case~0, we may exclude the case that $S=\Omega_{10}^-(2)$. Hence $m=3$ and $S\in \{\Omega_{10}^+(2),\Omega_{12}^-(2)\}$. Let $W'$ be a subspace of $W$ with $\dim(W')=2$ and consider $$\Delta=\{W''\le V\mid W'< W'', \dim(W'')=3, W''\textrm{ totally isotropic}\}.$$
Observe that, since each element $W'\in \Delta$ is totally isotropic, we have $W'\le W''\le (W'')^\perp\le (W')^\perp$. Now $\dim((W')^\perp/W')=6$ when $S=\Omega_{10}^+(2)$ and $\dim((W')^\perp/W')=8$ when $S=\Omega_{12}^-(2)$; moreover, the orthogonal form on $V$ induces a non-degenerate orthogonal form on $(W')^\perp/W'$. Furthermore, for each $W''\in \Delta$, $W''/W'$ is a $1$-dimensional isotropic subspace of $(W')^\perp/W'$. Consider $K=\stab_S((W')^\perp/W')$. An immediate application of Witt's lemma shows that $K$ induces the group $\mathrm{SO}((W')^\perp/W')$ in its natural action on $(W')^\perp/W'$. Now, by Case~0, when $S=\Omega_{12}^-(2)$, the group $\mathrm{SO}((W')^\perp/W')\cong\mathrm{SO}_8^-(2)$ admits a beautiful subset $\Lambda'$ of size $7$ for the action on the $1$-dimensional totally isotropic subspaces of the $8$-dimensional space $(W')^\perp/W'$. Therefore, a moment's thought yields that $\Lambda=\{W'\in \Delta\mid W'/W''\in \Lambda'\}$ is a beautiful subset of size $7$ for $S$ in its action on the $3$-dimensional totally isotropic subspaces. Assume then $S=\Omega_{10}^+(2)$ and recall that $\Omega_6^+(q)\cong \PSL_4(q)$. By Proposition~\ref{prop1}, the group $\mathrm{SO}((W')^\perp/W')\cong\mathrm{SO}_6^+(2)$ admits a beautiful subset $\Lambda'$. Therefore, a moment's thought yields that $\Lambda=\{W'\in \Delta\mid W'/W''\in \Lambda'\}$ is a beautiful subset  of $S$ in its action on the $3$-dimensional totally isotropic subspaces.

Now suppose that $m>3$.  In this case we take $U$ to be the subgroup of $S$ whose elements fix all elements of $\B$ except $e_m, e_1, e_2$ and $e_3$ and satisfy
 \[ 
 e_m \mapsto e_m+\alpha f_1+\beta f_2+\gamma f_3
 \]
for some $\alpha, \beta, \gamma \in \Fq$. Take $\Lambda= W^U$ and take $T$ to be a maximal split torus that preserves the subspace decomposition
\[
 \langle e_1, e_2, e_3 \rangle \oplus \langle f_1, f_2, f_3\rangle \oplus \bigoplus\limits_{\begin{array}{cc} v\in\B \\ v\not\in\{e_1, e_2, e_3, f_1, f_2, f_3\}\end{array}} \langle v\rangle
\]
and acts as a Singer cycle on $\langle e_1, e_2, e_3\rangle$. Then $U\rtimes T$ acts 2-transitively on $\Lambda$, a set of size $8$. Now define
\begin{align*}
 W_1 &= \bigcap\limits_{W'\in \Lambda} W' = \langle e_4,\dots, e_{m-1}\rangle; \\
 W_2 &= \langle W' \mid W'\in\Lambda\rangle = \langle e_1,\dots, e_m,f_1,f_2,f_3,f_m\rangle
\end{align*}
and observe that $S_\Lambda$ stabilizes $W_1$ and $W_2$. The action of $S_\Lambda$ on $\Lambda$ induces a homomorphism from some proper subgroup $H$ of  $\GL(W_2/W_1)=\GL_8(2)$ onto $S^\Lambda$. 

Suppose that $\Lambda$ is not beautiful. Then $S^\Lambda$ contains $\Alt(8)$ and so contains all double transpositions of $\Lambda$. Thus $S^\Lambda$ contains an element $g$ that interchanges the elements
\[
 \lambda_1=\langle e_1,e_2,e_3,f_m\rangle \textrm{ and }\lambda_2=\langle e_m+f_1, e_1+f_m, e_2,e_3\rangle,
\]
and fixes any 4 other elements of $\Lambda$. In particular we can choose $g$ so that it fixes
\begin{align*}
 \lambda_3&=\langle e_m+f_2, e_2+f_m, e_1, e_3\rangle, \,\,\,
 \lambda_4=\langle e_m+f_3, e_3+f_m, e_1, e_2\rangle, \\ \textrm{ and }
 \lambda_5&=\langle e_m+f_2+f_3, e_2+f_m, e_3+f_m, e_1\rangle.
 \end{align*}
But then $g$ fixes (setwise) $\lambda_3\cap\lambda_4\cap\lambda_5=\langle e_1\rangle$. This contradicts the fact that $g$ interchanges $\lambda_1$ and $\lambda_2$, and we are done.

\smallskip

\noindent\textsc{Case 2: }$M$ is the stabilizer of a non-degenerate subspace $W$ of $V$. We write $\dim(W)=m$.

\smallskip

\noindent Since in this case, $\stab_S(W)=\stab_S(W^\perp)$, we can assume that $\dim(W)<\frac{n}{2}$ except in the case $S=\Omega_n^-(q)$ with $n\equiv 0\pmod 4$. Note that if $\dim(W)$ is odd, then either $m=1$ or $q$ is odd (see \cite[Table 4.1.A]{kl} for a justification of these facts).

\smallskip

\noindent\textsc{Case 2a: } $q\geq 5$ and $W$ contains $v_1, w_1$ such that $\langle v_1, w_1\rangle $ is a hyperbolic line.\footnote{Note that, here and below, when we write that ``$\langle v_1,w_1\rangle$ is a hyperbolic line", we are also assuming that $v_1$ and $w_1$ are the usual distinguished elements, i.e. $v_1$ and $w_1$ are singular, and the scalar product $v_1\cdot w_1$ of $v_1$ with $w_1$ is $1$; in the literature this set of conditions is sometimes referred to by saying that ``$(v_1,w_1)$ is a hyperbolic pair''.}

\smallskip

\noindent Since $\dim(W)<\frac{n}{2}$ and $n\geq 7$, or $\dim(W)=n/2$ and $S=\Omega_{n}^-(q)$ with $n\ge 12$, with an elementary argument on $\dim(W^\perp)$ we find that $W^\perp$ contains elements $v_2, w_2$ where $\langle v_2, w_2\rangle$ is a hyperbolic line. In this case, define
\begin{align*}
W_0 &= \langle v_1,w_1\rangle^\perp \cap W; \\
 W_1 &= \langle W_0, w_1 \rangle; \\
 W_2 &= \langle W_0, w_1,v_1,v_2 \rangle.
\end{align*}
Now set
\[
 \Lambda = \left\langle W'\leq V \mid  \begin{array}{l} W'\textrm{ non-degenerate}, \dim(W')=\dim(W)=m, \\
 W_1\subseteq W' \subseteq W_2\end{array}\right\rangle.
\]
Observe that
\[
 W_1 = \bigcap\limits_{W'\in \Lambda} W' \textrm{ and } W_2= \langle W' \mid W'\in\Lambda\rangle.
\]
In particular, $S_\Lambda$ stabilizes $W_1$ and $W_2$. What is more there is a correspondence between elements of $\Lambda$ and 1-spaces in $W_2/W_1$ (via projection). Notice, though, that the 1-space $\langle v_2\rangle$ is not an element of $\Lambda$, since $\langle W_0, w_1, v_2\rangle$ is degenerate.

We conclude that $\Lambda$ is of size $q$, and that the action of $S_\Lambda$ on $\Lambda$ is permutation isomorphic to the action of a Borel subgroup $B$ of $\GL(W_2/W_1)\cong \GL_2(q)$ on the set of 1-spaces in $W_2/W_1$ that are not stabilized by $B$. In particular $\Lambda$ is a beautiful subset.

\smallskip
 
\noindent\textsc{Case 2b: } $q\in\{3,4\}$ and $W$ contains $v_1, w_1$ such that $\langle v_1,w_1\rangle$ is a hyperbolic line.

\smallskip

\noindent By~\textsc{Case 0}, we may assume that $n\geq 8$. If $S=\Omega_8^-(q)$ and $m=4$, then we choose $W$ to be of type ${\rm O}_4^-$. A quick case-by-case analysis reveals that in all cases $W^\perp$ contains elements $v_2, v_3, w_2, w_3 $ where $\langle v_2, w_2 \rangle$ and $\langle v_3, w_3\rangle$ are orthogonal hyperbolic lines. Extend $\{v_1, w_1, v_2, w_2, v_3, w_3\}$ to a hyperbolic basis $\mathcal{C}$. Now we take $U$ to be the subgroup of $S$ whose elements fix all elements of $\C$ except $v_1, w_2$ and $w_3$ and satisfy
 \[ 
 v_1 \mapsto v_1+\alpha v_2+\beta v_3
 \]
for some $\alpha, \beta \in \Fq$. Take $\Lambda= W^U$ and take $T$ to be a maximal torus that preserves the decomposition
\[
 \langle v_2, v_3 \rangle \oplus \langle w_2, w_3\rangle \oplus \bigoplus\limits_{\begin{array}{cc} v\in\C \\ v\not\in\{v_2, w_2, v_3, w_3\}\end{array}} \langle v\rangle
\]
and acts as a Singer cycle on $\langle v_2,v_3\rangle$. Applying Lemma~\ref{l: singers} to the space $\langle v_1,w_1\rangle^\perp$, we see that such a $T$ exists and, indeed, that it contains a subtorus that acts trivially on $\langle v_1\rangle$ and acts as a Singer cycle on $\langle v_2,v_3\rangle$. We conclude that $U\rtimes T$ acts $2$-transitively on $\Lambda$, a set of size $q^2$.  Now define
\begin{align*}
W_0 &=\langle v_1,w_1\rangle^\perp \cap W; \\
 W_1 &= \bigcap\limits_{W'\in \Lambda} W' = \langle W_0, w_1\rangle; \\
 W_2 &= \langle W' \mid W'\in\Lambda\rangle = \langle  W_0, w_1, v_1,v_2,v_3\rangle
\end{align*}
and observe that $S_\Lambda$ stabilizes $W_1$ and $W_2$, and the action of $S_\Lambda$ on $\Lambda$ induces a homomorphism from some subgroup $H$ of  $\GL(W_2/W_1)\cong\GL_3(q)$ onto $S^\Lambda$. If $\Lambda$ is not beautiful, then we obtain an epimorphism from a subgroup of $\GL_3(q)$ onto $\Alt(q^2)$, which is impossible (see, for instance, \cite[Proposition 5.3.7]{kl}).

\smallskip

\smallskip
 
\noindent\textsc{Case 2c: } $q=2$ and $W$ contains $v_1, w_1$ such that $\langle v_1,w_1\rangle$ is a hyperbolic line.

\smallskip

\noindent By~\textsc{Case 0}, we may assume that $n\geq 10$. Note, moreover, that since $q$ is even, $\dim(W)$ may be assumed to be even. Suppose, first, that either $S=\Omega_{10}^+(2)$ and $W$ is of type ${\rm O}_4^-(2)$, or else $S=\Omega_{10}^-(2)$ and $W$ is of type ${\rm O}_4^+(2)$. In either case $W^\perp$ is of type ${\rm O}_6^-(2)$. Write $S=\Omega_{10}^\varepsilon(2)$ with $\varepsilon\in\{+,-\}$. Let $W'$ be a $2$-dimensional non-degenerate subspace of $V$ of type $\mathrm{O}_2^{-\varepsilon}(2)$ and set $H=\stab_S(W')$. Clearly, $H$ is a subspace stabilizer of type $\mathrm{O}_2^{-\varepsilon }(2)\times \mathrm{O}_8^{-}(2)$ and $H$ induces the group $\mathrm{O}_8^{-}(2)$ on $(W')^\perp$. From the footnote~\ref{fn123}, the action of $H$ on the $2$-dimensional non-degenerate subspaces of $(W')^\perp$ with stabilizers of type $\mathrm{O}_2^+(2)\times \mathrm{O}_6^{-}(2)$ admits a beautiful subset $\Lambda'$ of size $5$ with the property that $\bigcap_{W''\in \Lambda'}W''=0$. Now, set $\Lambda:=\{W'\oplus W''\mid W''\in\Lambda'\}$ and observe that $\Lambda$ consists of five $4$-dimensional non-degenerate subspaces of $V$ of type $\mathrm{O}_4^{-\varepsilon}(2)$, that is, the stabilizer of an element of $\Lambda$ in $S$ is a subspace stabilizer of type $\mathrm{O}_4^{-\varepsilon}(2)\times\mathrm{O}_6^-(2)$. Set $K:=\stab_S(\Lambda)$. Then $K$ stabilizes $\cap_{W\in \Lambda}W=W'$ and hence $K\leq H$. From this it immediately follows that the action of $K$ on $\Lambda$ is permutation isomorphic to the action of $H$ on $\Lambda'$ and hence $\Lambda$ is a beautiful subset of size $5$.

From here on we exclude the two cases just discussed. What is more, if $S=\Omega_{12}^-(2)$, then we choose $W$ to be of type ${\rm O}_6^-$. Now a quick case-by-case analysis reveals that in all cases $W^\perp$ contains elements $v_2, v_3, v_4, w_2, w_3, w_4 $ where $\langle v_2, w_2 \rangle$, $\langle v_3, w_3\rangle$  and $\langle v_4, w_4\rangle$ are mutually orthogonal hyperbolic lines. Extend the set $\{v_1, w_1, v_2, w_2, v_3, w_3, v_4, w_4\}$ to a hyperbolic basis $\mathcal{C}$. Now we take $U$ to be the subgroup of $S$ whose elements fix all elements of $\C$ except $v_1, w_2, w_3$ and $w_4$ and satisfy
 \[ 
 v_1 \mapsto v_1+\alpha v_2+\beta v_3+\gamma v_4
 \]
for some $\alpha, \beta,\gamma \in \Fq$. Take $\Lambda= W^U$ and take $T$ to be a maximal torus that preserves the decomposition
\[
 \langle v_2, v_3,v_4 \rangle \oplus \langle w_2, w_3,w_4\rangle \oplus \bigoplus\limits_{\begin{array}{cc} v\in\C \\ v\not\in\{v_2, w_2, v_3, w_3, v_4, w_4\}\end{array}} \langle v\rangle
\]
and acts as a Singer cycle on $\langle v_2,v_3,v_4\rangle$. Applying Lemma~\ref{l: singers} to the space $\langle v_1,w_1\rangle^\perp$, we see that such a $T$ exists and, indeed, that it contains a subtorus that acts trivially on $\langle v_1\rangle$ and acts as a Singer cycle on $\langle v_2,v_3\rangle$. We conclude that $U\rtimes T$ acts $2$-transitively on $\Lambda$, a set of size $8$.  Now define
\begin{align*}
W_0 &=\langle v_1,w_1\rangle^\perp \cap W; \\
 W_1 &= \bigcap\limits_{W'\in \Lambda} W' = \langle W_0, w_1\rangle; \\
 W_2 &= \langle W' \mid W'\in\Lambda\rangle = \langle  W_0, w_1, v_1,v_2,v_3,v_4\rangle
\end{align*}
and observe that $S_\Lambda$ stabilizes $W_1$ and $W_2$, and the action of $S_\Lambda$ on $\Lambda$ induces a homomorphism from some subgroup $H$ of  $\GL(W_2/W_1)\cong\GL_4(2)$ onto $S^\Lambda$. (Note that $H$ is a \emph{proper} subgroup since, for instance, it stabilizes the $3$-space $\langle v_2,v_3,v_4,W_1\rangle/W_1$.) If $\Lambda$ is not beautiful, then we obtain an epimorphism from a proper subgroup of $\GL_4(2)$ onto $\Alt(8)$, which is impossible.

\smallskip

\noindent\textsc{Case 2d: } $q$ is odd and $W$ does not contain a hyperbolic line.

\smallskip

\noindent Here $n\geq 7$ and $W$ is anisotropic. Since $\dim(W)\leq 2$, we know that $W^\perp$ contains elements $v_2, v_3, w_2, w_3 $ where $\langle v_2, w_2 \rangle$ and $\langle v_3, w_3\rangle$ are orthogonal hyperbolic lines. Extend $\{v_2, w_2, v_3, w_3\}$ to a basis $\C_1$, for $W^\perp$. Now extend $\C_1$ to a basis $\C$ for $V$ by adding elements that form a basis for $W$ -- so we add either one or two elements. If $|W\cap\C|=1$, then we write $W\cap\C=\{x\}$; if $|W\cap\C|=2$, then we write $W\cap\C=\{x,y\}$ and require that $x\cdot y=1$. In either case we write $\gamma=x\cdot x$. 

If $\dim(W)=1$, then we define $U$ to be the subgroup of $S$ whose elements fix all elements of $\C$ except $x$, $v_2$ and $v_3$ and satisfy
 \[
  x \mapsto x+\alpha w_2+\beta w_3,\,\, 
  v_2 \mapsto v_2-\frac{\alpha}{\gamma} x - \frac{\alpha^2}{2\gamma} w_2,\,\,
  v_3 \mapsto v_3-\frac{\beta}{\gamma} x  - \frac{\beta^2}{2\gamma} w_3,
 \]
for some $\alpha, \beta \in \Fq$. If $\dim(W)=2$, then we define $U$ to be the subgroup of $S$ whose elements fix all elements of $\C$ except $x$, $y$, $v_2$ and $v_3$ and satisfy
 \begin{alignat*}{2}
  x &\mapsto x+\alpha w_2+\beta w_3, \,\,   &y \mapsto y+\frac{\alpha}{\gamma} w_2+\frac{\beta}{\gamma} w_3,\\
  v_2 &\mapsto v_2-\frac{\alpha}{\gamma} x - \frac{\alpha^2}{2\gamma} w_2, \,\,
  &v_3\mapsto v_3-\frac{\beta}{\gamma} x  - \frac{\beta^2}{2\gamma} w_3,
 \end{alignat*}
for some $\alpha, \beta \in \Fq$. In both cases, take $\Lambda= W^U$ and take $T$ to be a maximal torus that preserves the decomposition
\[
 \langle v_2, v_3 \rangle \oplus \langle w_2, w_3\rangle \oplus \bigoplus\limits_{\begin{array}{cc} v\in\C \\ v\not\in\{v_2, w_2, v_3, w_3\}\end{array}} \langle v\rangle
\]
and acts as a Singer cycle on $\langle v_2,v_3\rangle$. Applying Lemma~\ref{l: singers} to the space $W^\perp$, we see that such a $T$ exists and, indeed, that it contains a subtorus that acts trivially on $\langle x\rangle$ and acts as a Singer cycle on $\langle v_2,v_3\rangle$. We conclude that $U\rtimes T$ acts $2$-transitively on $\Lambda$, a set of size $q^2$. Now define
\[ W_2 = \langle W' \mid W'\in\Lambda\rangle = 
\begin{cases}
  \langle x, w_2, w_3\rangle, &\textrm{if }|U|=q; \\
  \langle x, y, w_2, w_3\rangle, &\textrm{if }|U|=q^2.
  \end{cases}
\]
Observe that $S_\Lambda$ stabilizes $W_2$, and the action of $S_\Lambda$ on $\Lambda$ induces a homomorphism from some subgroup $H$ of  $\GL(W_2)$ (which is isomorphic to either $\GL_3(q)$ or $\GL_4(q)$) onto $S^\Lambda$. We conclude that $S^\Lambda$ does not contain $\Alt(q^2)$, and so $\Lambda$ is a beautiful subset.

\smallskip

\noindent\textsc{Case 2e: } $q=2^a$ and $W$ does not contain a hyperbolic line.

\smallskip

\noindent In this case $W$ is either anisotropic of dimension $1$, or is of type ${\rm O}_2^-$. In the former case, Witt's Lemma implies that $W$ lies inside a subspace of type ${\rm O}_2^-$ which we call $X$. In the latter case we set $X=W$. So, in both cases, $X$ is of type ${\rm O}_2^-$.

Referring to \cite[Lemma 2.5.2]{kl} we write $X=\langle x,y\rangle$ where $(Q(x), Q(y), x\cdot y) = (1, \zeta, 1)$ for $Q$ the quadratic form associated to $S$, and $\zeta$ an element of $\Fq$ such that $x^2+x+\zeta$ is irreducible over $\Fq$. If $W$ is a proper subspace of $X$, then we choose a labeling so that $x\in W$; thus, in any case, $x\in W$.

Suppose, first, that $a\geq 3$ (i.e.\ $q\geq 8$). Since $n\geq 8$, we know that there exist $v_2, w_2\in X^{\perp}$ such that $\langle v_2, w_2\rangle$ is a hyperbolic line. Extend the set $\{v_2, w_2\}$ to a hyperbolic basis $\C_1$ for $X^\perp$, and then extend this to a basis $\C$ for $V$ by adding the elements $x$ and $y$. Now define $U$ to be the subgroup of $S$ whose elements fix all elements of $\C$ except $x$ and $v_2$ and satisfy
\[  x \mapsto x+\alpha w_2, \,\,\,\,\,
  v_2 \mapsto v_2 + \alpha y+ \zeta \alpha^2 w_2, 
 \]
for some $\alpha \in \Fq$. Take $\Lambda= W^U$ and take $T$ to be the stabilizer of the direct sum decomposition
$ \bigoplus\limits_{v\in\C} \langle v\rangle.$
Then $U\rtimes T$ acts $2$-transitively on $\Lambda$, a set of size $q$. Now define
\begin{align*}
 W_1 &= \bigcap\limits_{W'\in \Lambda} W'; \\
 W_2 &= \langle W' \mid W'\in\Lambda\rangle = \langle W_1, x, w_2\rangle
\end{align*}
and observe that $S_\Lambda$ stabilizes $W_1$ and $W_2$, and the action of $S_\Lambda$ on $\Lambda$ induces a homomorphism from some subgroup $H$ of  $\GL(W_2/W_1)\cong\GL_2(q)$ onto $S^\Lambda$. We conclude that $S^\Lambda$ does not contain $\Alt(q)$, and so $\Lambda$ is a beautiful subset.

Suppose, next, that $q=4$. Since $n\geq 8$, we know that there exist $v_2, w_2, v_3, w_3\in X^{\perp}$ such that $\langle v_2, w_2\rangle$ and $\langle v_3, w_2\rangle$ are orthogonal hyperbolic lines. Extend the set $\{v_2, w_2, v_3, w_3\}$ to a hyperbolic basis $\C_1$ for $X^\perp$, and then extend this to a basis $\C$ for $V$ by adding the elements $x$ and $y$. Now define $U$ to be the subgroup of $S$ whose elements fix all elements of $\C$ except $x, v_2$ and $v_3$ and satisfy
\[  x \mapsto x+\alpha w_2 + \beta w_3, \,\,\,\,\,
  v_2 \mapsto v_2 + \alpha y+ \zeta \alpha^2 w_2,\,\,\,\,\,
  v_3 \mapsto v_3 + \beta y+ \zeta \beta^2 w_3, 
 \]
for some $\alpha, \beta \in \Fq$. Take $\Lambda= W^U$ and take $T$ to be a maximal torus of $S$ that stabilizes the decomposition
\[
 \langle v_2, v_3 \rangle \oplus \langle w_2, w_3\rangle \oplus \bigoplus\limits_{\begin{array}{cc} v\in\C \\ v\not\in\{v_2, w_2, v_3, w_3\}\end{array}} \langle v\rangle
\]
and acts as a Singer cycle on $\langle v_2,v_3\rangle$. Now $U\rtimes T$ acts $2$-transitively on $\Lambda$, a set of size $16$. As before, we set
\begin{align*}
 W_1 &= \bigcap\limits_{W'\in \Lambda} W'; \\
 W_2 &= \langle W' \mid W'\in\Lambda\rangle = \langle W_1, x, w_2, w_3\rangle.
\end{align*}
 Now the action of $S_\Lambda$ on $\Lambda$ induces a homomorphism from some subgroup $H$ of  $\GL(W_2/W_1)\cong\GL_3(4)$ onto $S^\Lambda$. We conclude that $S^\Lambda$ does not contain $\Alt(16)$, and so $\Lambda$ is a beautiful subset of size $q^2=16$.

Finally, when $q=2$, we proceed similarly and we use the fact that, by \textsc{Case 0},  $n\geq 10$. With similar definitions, $U$ is the subgroup whose elements fix all elements of $\C$ except $x, v_2, v_3$ and $v_4$ and satisfy
\begin{align*}  x &\mapsto x+\alpha w_2 + \beta w_3 + \gamma w_4, \,\,\,\,\,
  &v_2 \mapsto v_2 + \alpha y+ \zeta \alpha^2 w_2, \\
  v_3 &\mapsto v_3 + \beta y+ \zeta \beta^2 w_3, \,\,\,\,\,
  &v_4 \mapsto v_4 + \gamma y+ \zeta \gamma^2 w_4, 
 \end{align*}
for some $\alpha, \beta, \gamma \in \Fq$. Take $\Lambda= W^U$ and take $T$ to be a maximal torus of $S$ that stabilizes the decomposition
\[
\langle v_2, v_3,v_4 \rangle \oplus \langle w_2, w_3, w_4\rangle \oplus \bigoplus\limits_{\begin{array}{cc} v\in\C \\ v\not\in\{v_2, w_2, v_3, w_3,v_4,w_4\}\end{array}} \langle v\rangle
\]
and acts as a Singer cycle on $\langle v_2,v_3,v_4\rangle$. Now $U\rtimes T$ acts $2$-transitively on $\Lambda$, a set of size $8$. As before, we set
\begin{align*}
 W_1 &= \bigcap\limits_{W'\in \Lambda} W'; \\
 W_2 &= \langle W' \mid W'\in\Lambda\rangle = \langle W_1, x, w_2, w_3,w_4\rangle.
\end{align*}
 Now the action of $S_\Lambda$ on $\Lambda$ induces a homomorphism from some proper subgroup $H$ of $\GL(W_2/W_1)\cong\GL_4(2)$ onto $S^\Lambda$. (Note that $H$ is a \emph{proper} subgroup since, for instance, it stabilizes the $3$-space $\langle w_2,w_3,w_4\rangle+W_1$.) We conclude that $S^\Lambda$ does not contain $\Alt(8)$, and so $\Lambda$ is a beautiful subset of size $8$.
 \end{proof}

\subsection{Proof of Theorem~B}

In view of Propositions~\ref{prop1},~\ref{prop2},~\ref{prop3} and~\ref{prop4}, to prove Theorem~B, we must deal with all of those actions listed in Table~\ref{t: c1 sln}. Note that this table lists actions for the group $S$, and we must deal with actions for each almost simple group $G$ with socle isomorphic to $S/Z(S)$. The relevant actions are listed in Table~\ref{t: remaining}. In compiling Table~\ref{t: remaining}, we make use of the information in Table~\ref{t: c1 sln}, of~\cite{ATLAS} and of the various isomorphisms between non-abelian simple groups (for instance $\PSL_2(4)\cong \PSL_2(5)\cong \Alt(5)$, $\PSL_2(9)\cong\Alt(6)\cong \Sp_4(2)'$ and $\PSU_4(2)\cong \mathrm{PSp}_4(3)$). Observe that under the isomorphism $\PSU_4(2)\cong \mathrm{PSp}_4(3)$, the stabilizer of a $1$-dimensional non-isotropic subspace in $\PSU_4(2)$ corresponds to the stabilizer of a $1$-dimensional totally isotropic subspace in $\mathrm{PSp}_4(3)$, in particular we have only one action of degree $40$ to consider in Table~\ref{t: remaining}.

\begin{center}
\begin{table}[!htbp]
\begin{tabular}{ccc}%{@{\makebox[3em][c]{\rownumber\space}} cc}
\toprule[1.5pt]
Line&Group & Degree of actions \\
\midrule[1.5pt]
1 &$\Alt(5)$ or $\Sym(5)$  & 5 \\
2&$\SL_3(2):2$ & 21, 28\\
3&$\mathrm{PSL}_3(3):2$& 52 \\
4& $\PSU_4(2)$ or $\PSU_4(2):2$ & 27, 40\\
5& $\PSU_5(2)$ or $\PSU_5(2):2$& 3520\\
6&  $\Alt(6)$ or $\Sym(6)$ & 15, 15 \\
7&$\mathrm{P}\Omega_7(3)$ and $\mathrm{PSO}_7(3)$ & $351$\\
 \bottomrule[1.5pt]
 \end{tabular}
\caption{Remaining actions.}\label{t: remaining}
\end{table}
\end{center}
Line~$1$ of Table~\ref{t: remaining} yields the example given in Theorem~B. All other examples with index at most 100 can be dealt with easily via computer; indeed, their non-binariness, and more, has been confirmed by Wiscons~\cite{wiscons2}.

This leaves Line~5 and~7. Using \texttt{magma} we can see that the primitive groups arising from Lines~5 and~7 are not binary: for both lines, we have constructed the permutation representations of degree $3520$ and $351$ (respectively) and we have found two non-binary witnesses of length $3$.

\bibliography{paper6}{}
\bibliographystyle{plain} 
\end{document}